\documentclass[reqno]{amsart}
\usepackage{amssymb}
\usepackage{amsmath}
\usepackage{amsfonts}

\newtheorem{theorem}{Theorem}[section]
\theoremstyle{plain}

\newtheorem{lemma}{Lemma}[section]

\newtheorem{remark}{Remark}[section]

\numberwithin{equation}{section}

\begin{document}
\title[Existence of the global attractor ]{Existence of the global attractor
for the plate equation with nonlocal nonlinearity in $%
\mathbb{R}
^{n}$}
\author{ Azer Khanmamedov \ }
\address{{\small Department of Mathematics, Faculty of Science, Hacettepe
University, Beytepe 06800, Ankara, Turkey}}
\email{azer@hacettepe.edu.tr}
\author{ \ Sema Simsek }
\address{{\small Department of Mathematics, Faculty of Science, Hacettepe
University, Beytepe 06800, Ankara, Turkey}}
\email{semasimsek@hacettepe.edu.tr}
\subjclass[2000]{35B41, 35G20, 74K20}
\keywords{plate equation, global attractor}

\begin{abstract}
We consider Cauchy problem for the semilinear plate equation with nonlocal
nonlinearity. Under mild conditions on the damping coefficient, we prove
that the semigroup generated by this problem possesses a global attractor.
\end{abstract}

\maketitle

\section{Introduction}

\bigskip In this paper, we study the long-time behavior of the solutions for
the following plate equation with localized damping and nonlocal
nonlinearity in terms of global attractors:%
\begin{equation}
u_{tt}+\Delta ^{2}u+\alpha (x)u_{t}+\lambda u+f(\left\Vert u\left( t\right)
\right\Vert _{L^{p}\left(
\mathbb{R}
^{n}\right) })\left\vert u\right\vert ^{p-2}u=h\left( x\right) \text{, \ \ \
}(t,x)\in (0,\infty )\times
\mathbb{R}
^{n}\text{.}  \tag{1.1}
\end{equation}

Plate equations have been investigated for many years due to their
importance in some physical areas such as vibration and elasticity theories
of solid mechanics. For instance, in the case when $f\left( \cdot \right) $
is identically constant, equation (1.1) becomes an equation with local
polynomial nonlinearity which arises in aeroelasticity modeling (see, for
example, [7], [8]), whereas in the case when $p=2$, the nonlinearity $%
f(\left\Vert u\right\Vert _{L^{p}\left(
\mathbb{R}
^{n}\right) })\left\vert u\right\vert ^{p-2}u$ appears in the models of
Kerr-like medium (see [15], [21]).

The study of the long-time dynamics of evolution equations has become an
outstanding area during the recent decades. As it is well known, the
attractors can be used as a tool to describe the long-time dynamics of these
equations. In particular, there have been many works on the investigation of
the attractors for the plate equations over the last few years. For the
attractors of the plate equations with local and nonlocal nonlinearities in
bounded domains, we refer to [3], [5], [14], [16-19] and [22]. In the case
of unbounded domains, owing to the lack of Sobolev compact embedding
theorems, there are difficulties in applying the methods given for bounded
domains. In order to overcome these difficulties, the authors of [9-10],
[13] and [23] established the uniform tail estimates for the plate equations
with local nonlinearities and then used the weak continuity of the nonlinear
source operators.

In the case when the domain is unbounded and the equation includes nonlocal
nonlinearity, an additional obstacle occurs. For equation (1.1), this
obstacle is caused by the operator defined by $F\left( u\right)
:=f(\left\Vert u\right\Vert _{L^{p}\left(
\mathbb{R}
^{n}\right) })\left\vert u\right\vert ^{p-2}u$ , because the operator $F$,
besides being not compact, is not also weakly continuous from $H^{2}\left(
\mathbb{R}
^{n}\right) $ to $L^{2}\left(
\mathbb{R}
^{n}\right) $. This situation does not allow us to apply the standard
splitting method and the energy method devised in [2]. Recently in [1], the
obstacle mentioned above is handled for the nonlinearity $f(\left\Vert
\nabla u\right\Vert _{L^{2}\left(
\mathbb{R}
^{n}\right) })\Delta u$ by using compensated compactness method introduced
in [11]. In that paper, the strictly positivity condition on the damping
coefficient $\alpha \left( \cdot \right) $ is critically used. In the
present paper, we replace this condition with the weaker conditions (see
(2.3), (2.4)), and by using effectiveness of the dissipation for large
enough $x$, we prove the existence of the global attractor which equals the
unstable manifold of the set of stationary points.

The paper is organized as follows: In Section 2, we give the statement of
the problem and the main result. In Section 3, we firstly prove two
auxiliary lemmas and then establish the asymptotic compactness of the
solution, which together with the presence of the strict Lyapunov function
leads to the existence of the global attractor.

\section{\ Statement of the problem and the main result}

We consider the following initial value problem%
\begin{equation}
u_{tt}+\Delta ^{2}u+\alpha (x)u_{t}+\lambda u+f(\left\Vert u\left( t\right)
\right\Vert _{L^{p}\left(
\mathbb{R}
^{n}\right) })\left\vert u\right\vert ^{p-2}u=h\left( x\right) \text{, \ \ \
}(t,x)\in (0,\infty )\times
\mathbb{R}
^{n}\text{,}  \tag{2.1}
\end{equation}%
\begin{equation}
u(0,x)=u_{0}(x)\text{, \ }u_{t}(0,x)=u_{1}(x)\text{, \ \ }x\in
\mathbb{R}
^{n}\text{,}  \tag{2.2}
\end{equation}%
where $\lambda >0,$ $h\in L^{2}\left(
\mathbb{R}
^{n}\right) ,$ $p\geq 2,$ $p\left( n-4\right) \leq 2n-4$ and the functions $%
\alpha \left( \cdot \right) $, $f\left( \cdot \right) $ satisfy the
following conditions:%
\begin{equation}
\alpha \in L^{\infty }(%
\mathbb{R}
^{n})\text{, }\alpha (\cdot )>0\text{ \ a.e. in }%
\mathbb{R}
^{n}\text{,}  \tag{2.3}
\end{equation}%
\begin{equation}
\alpha (\cdot )\geq \alpha _{0}>0\text{ a.e. in }\left\{ x\in
\mathbb{R}
^{n}:\left\vert x\right\vert \geq r_{0}\right\} \text{, for some }r_{0}\text{%
,}  \tag{2.4}
\end{equation}%
\begin{equation}
f\in C^{1}(%
\mathbb{R}
^{+})\text{, \ }f\left( \cdot \right) \geq 0\text{.}  \tag{2.5}
\end{equation}

Applying the semigroup theory (see [4, p.56-58]) and repeating the arguments
done in the introduction of [1], one can prove the following well-posedness
result.

\begin{theorem}
Assume that the conditions (2.3)-(2.5) hold. Then for every $T>0$ and $%
\left( u_{0},u_{1}\right) \in H^{2}\left(
\mathbb{R}
^{n}\right) \times L^{2}\left(
\mathbb{R}
^{n}\right) ,$ problem (2.1)-(2.2) has a unique weak solution $u\in C\left(
[0,T];H^{2}\left(
\mathbb{R}
^{n}\right) \right) \cap $\newline
$C^{1}\left( [0,T];L^{2}\left(
\mathbb{R}
^{n}\right) \right) $ which satisfies the energy equality
\begin{equation*}
E_{%
\mathbb{R}
^{n}}\left( u\left( t\right) \right) +\frac{1}{p}F\left( \left\Vert u\left(
t\right) \right\Vert _{L^{p}\left(
\mathbb{R}
^{n}\right) }^{p}\right) -\int\limits_{%
\mathbb{R}
^{n}}h\left( x\right) u\left( t,x\right) dx+\int\limits_{s}^{t}\int\limits_{%
\mathbb{R}
^{n}}\alpha \left( x\right) \left\vert u_{t}\left( \tau ,x\right)
\right\vert ^{2}dxd\tau
\end{equation*}%
\begin{equation}
\text{ }=E_{%
\mathbb{R}
^{n}}\left( u\left( s\right) \right) +\frac{1}{p}F\left( \left\Vert u\left(
s\right) \right\Vert _{L^{p}\left(
\mathbb{R}
^{n}\right) }^{p}\right) -\int\limits_{%
\mathbb{R}
^{n}}h\left( x\right) u\left( s,x\right) dx\text{, \ }\forall t\geq s\geq 0%
\text{,}  \tag{2.6}
\end{equation}%
where $F\left( z\right) =\int\limits_{0}^{z}f\left( \sqrt[p]{s}\right) ds$,
for all $z\in
\mathbb{R}
^{+}$ and $E_{\Omega }\left( u\left( t\right) \right) =\frac{1}{2}%
\int\limits_{\Omega }(\left\vert u_{t}\left( t,x\right) \right\vert
^{2}+\left\vert {\small \Delta u}\left( t,x\right) \right\vert ^{2}+\lambda
\left\vert {\small u}\left( t,x\right) \right\vert ^{2})dx$, for subset $%
\Omega \subset
\mathbb{R}
^{n}$. Moreover, if $\left( u_{0},u_{1}\right) \in H^{4}\left(
\mathbb{R}
^{n}\right) \times H^{2}\left(
\mathbb{R}
^{n}\right) $, then $u\in C\left( [0,T];H^{4}\left(
\mathbb{R}
^{n}\right) \right) \cap C^{1}\left( [0,T];H^{2}\left(
\mathbb{R}
^{n}\right) \right) $.

In addition, if $v,w\in C\left( [0,T];H^{2}\left(
\mathbb{R}
^{n}\right) \right) \cap C^{1}\left( [0,T];L^{2}\left(
\mathbb{R}
^{n}\right) \right) $ are the weak \ solutions to (2.1)-(2.2) with initial
data $\left( v_{0},v_{1}\right) \in H^{2}\left(
\mathbb{R}
^{n}\right) \times L^{2}\left(
\mathbb{R}
^{n}\right) $ and $\left( w_{0},w_{1}\right) \in H^{2}\left(
\mathbb{R}
^{n}\right) \times L^{2}\left(
\mathbb{R}
^{n}\right) $, then
\begin{equation*}
\left\Vert v(t)-w(t)\right\Vert _{H^{2}\left(
\mathbb{R}
^{n}\right) }+\left\Vert v_{t}(t)-w_{t}(t)\right\Vert _{L^{2}(%
\mathbb{R}
^{n})}
\end{equation*}%
\begin{equation*}
\leq c(T,\widetilde{r})\left( \left\Vert v_{0}-w_{0}\right\Vert
_{H^{2}\left(
\mathbb{R}
^{n}\right) }+\left\Vert v_{1}-w_{1}\right\Vert _{L^{2}\left(
\mathbb{R}
^{n}\right) }\right) \text{, \ }\forall t\in \lbrack 0,T]\text{,}
\end{equation*}%
where $c:R_{+}\times R_{+}\rightarrow R_{+}$ \ is a nondecreasing function
with respect to each variable and $\widetilde{r}=\max \left\{ \left\Vert
(v_{0},v_{1})\right\Vert _{H^{2}\left(
\mathbb{R}
^{n}\right) \times L^{2}\left(
\mathbb{R}
^{n}\right) },\left\Vert (w_{0},w_{1})\right\Vert _{H^{2}\left(
\mathbb{R}
^{n}\right) \times L^{2}\left(
\mathbb{R}
^{n}\right) }\right\} $.
\end{theorem}

Thus, according to Theorem 2.1, by the formula $\left( u(t),u_{t}(t)\right)
=S\left( t\right) \left( u_{0},u_{1}\right) $, problem (2.1)-(2.2) generates
a strongly continuous semigroup $\left\{ S\left( t\right) \right\} _{t\geq
0} $ in $H^{2}\left(
\mathbb{R}
^{n}\right) \times L^{2}\left(
\mathbb{R}
^{n}\right) $, where $u(t,\cdot )$ is the weak solution of (2.1)-(2.2),
determined by Theorem 2.1, with initial data $\left( u_{0},u_{1}\right) $.

Now, we are in a position to state our main result.

\begin{theorem}
Under the conditions (2.3)-(2.5), the semigroup $\left\{ S\left( t\right)
\right\} _{t\geq 0}$ generated by the problem (2.1)-(2.2) possesses a global
attractor $\mathcal{A}$ in $H^{2}\left(
\mathbb{R}
^{n}\right) \times L^{2}\left(
\mathbb{R}
^{n}\right) $ and $\mathcal{A=M}^{u}\left( \mathcal{N}\right) $. Here $%
\mathcal{M}^{u}\left( \mathcal{N}\right) $ is unstable manifold emanating
from the set of stationary points $\mathcal{N}$ (for definition, see [6,
p.359]).

\begin{remark}
We note that by using the method of this paper, one can prove the existence
of the global attractors for the initial boundary value problems%
\begin{equation*}
\left\{
\begin{array}{c}
u_{itt}+(-\Delta )^{i}u_{i}+\alpha (x)u_{it}+\lambda u_{i}+f(\left\Vert
u_{i}\left( t\right) \right\Vert _{L^{p}\left(
\mathbb{R}
^{n}\right) })\left\vert u_{i}\right\vert ^{p_{i}-2}u_{i}=h\left( x\right)
\text{, \ \ }(t,x)\in (0,\infty )\times \Omega \text{,} \\
u_{i}(t,x)=\left( \frac{\partial }{\partial \nu }\right) ^{i-1}u_{i}(t,x)=0%
\text{, \ \ \ \ \ \ \ \ \ \ \ \ \ \ \ \ \ \ \ \ \ \ \ \ \ \ \ \ \ \ \ \ \ \
\ \ \ \ \ \ \ \ \ \ \ \ \ \ }(t,x)\in (0,\infty )\times \partial \Omega
\text{,} \\
u_{i}(0,x)=u_{0i}(x)\text{, \ \ \ \ }u_{it}(0,x)=u_{1i}(x)\text{, \ \ \ \ \
\ \ \ \ \ \ \ \ \ \ \ \ \ \ \ \ \ \ \ \ \ \ \ \ \ \ \ \ \ \ \ \ \ \ \ \ \ \
\ \ \ \ \ \ \ \ \ \ \ }x\in \Omega \text{,}%
\end{array}%
\right.
\end{equation*}%
where $\Omega \subset
\mathbb{R}
^{n}$ is an unbounded domain with smooth boundary, $\nu $ is outer unit
normal vector, $\lambda >0$, $h\in L^{2}\left( \Omega \right) $, $p_{i}\geq 2
$, $p_{i}\left( n-2i\right) \leq 2n-2i$, $i=1,2$, the function $f(\cdot )$
satisfies the condition (2.5) and the damping coefficient $\alpha (\cdot )$
satisfies the following conditions%
\begin{equation*}
\alpha \in L^{\infty }(\Omega )\text{, \ \ }\alpha (\cdot )>0\text{, \ a.e.
in \ }\Omega \text{,}
\end{equation*}%
\begin{equation*}
\alpha \left( \cdot \right) \geq \alpha _{0}>0\text{, \ a.e. in \ }\omega
\text{, \ for some }\omega \subset \Omega \text{, such that }
\end{equation*}%
$\omega $ is the union of a neighbourhood of the boundary $\partial \Omega $
and $\left\{ x\in \Omega :\left\vert x\right\vert \geq r_{0}\right\} $, for
some $r_{0}$.
\end{remark}
\end{theorem}

\section{Proof of Theorem 2.2}

We start with the following lemmas.

\begin{lemma}
Assume that the condition (2.5) holds. Also, assume that\ the sequence $%
\left\{ v_{m}\right\} _{m=1}^{\infty }$ is weakly star convergent in $%
L^{\infty }\left( 0,\infty ;H^{2}\left(
\mathbb{R}
^{n}\right) \right) $, the sequence $\left\{ v_{mt}\right\} _{m=1}^{\infty }$
is bounded in $L^{\infty }\left( 0,\infty ;L^{2}\left(
\mathbb{R}
^{n}\right) \right) $ and the sequence $\left\{ \left\Vert v_{m}\left(
t\right) \right\Vert _{L^{p}\left(
\mathbb{R}
^{n}\right) }\right\} _{m=1}^{\infty }$ is convergent, for all $t\geq 0$.
Then, for all $r>0$%
\begin{equation*}
\underset{m\rightarrow \infty }{\lim }\text{ }\underset{l\rightarrow \infty }%
{\lim \sup }\left\vert \int\limits_{0}^{t}\int\limits_{B\left( 0,r\right)
}\tau \left( f\left( \left\Vert v_{m}\left( \tau \right) \right\Vert
_{L^{p}\left(
\mathbb{R}
^{n}\right) }\right) \left\vert v_{m}\left( \tau ,x\right) \right\vert
^{p-2}v_{m}(\tau ,x)\right. \right.
\end{equation*}%
\begin{equation*}
\left. \left. -f\left( \left\Vert v_{l}\left( \tau \right) \right\Vert
_{L^{p}\left(
\mathbb{R}
^{n}\right) }\right) \left\vert v_{l}(\tau ,x)\right\vert ^{p-2}v_{l}(\tau
,x)\right) \left( v_{mt}\left( \tau ,x\right) -v_{lt}\left( \tau ,x\right)
\right) dxd\tau \right\vert =0,\text{ \ }\forall t\geq 0,
\end{equation*}%
where $B\left( 0,r\right) =\left\{ x\in
\mathbb{R}
^{n}:\left\vert x\right\vert <r\right\} $.
\end{lemma}

\begin{proof}
Denote $f_{\varepsilon }\left( u\right) $ $=\left\{
\begin{array}{c}
f\left( u\right) ,\text{ \ \ \ }u\geq \varepsilon , \\
f\left( \varepsilon \right) ,\text{ }0\leq u<\varepsilon ,%
\end{array}%
\right. $ for $\varepsilon >0$. Then, we have%
\begin{equation*}
\left\vert f\left( \left\Vert v_{m}\left( \tau \right) \right\Vert
_{L^{p}\left(
\mathbb{R}
^{n}\right) }\right) -f_{\varepsilon }\left( \left\Vert v_{m}\left( \tau
\right) \right\Vert _{L^{p}\left(
\mathbb{R}
^{n}\right) }\right) \right\vert \leq \max_{0\leq s_{1},s_{2}\leq
\varepsilon }\left\vert f\left( s_{1}\right) -f\left( s_{2}\right)
\right\vert ,
\end{equation*}%
and consequently%
\begin{equation*}
\left\vert \int\limits_{0}^{t}\int\limits_{B\left( 0,r\right) }\tau \left(
f\left( \left\Vert v_{m}\left( \tau \right) \right\Vert _{L^{p}\left(
\mathbb{R}
^{n}\right) }\right) \left\vert v_{m}\left( \tau ,x\right) \right\vert
^{p-2}v_{m}(\tau ,x)-f\left( \left\Vert v_{l}\left( \tau \right) \right\Vert
_{L^{p}\left(
\mathbb{R}
^{n}\right) }\right) \left\vert v_{l}(\tau ,x)\right\vert ^{p-2}v_{l}(\tau
,x)\right) \right.
\end{equation*}%
\begin{equation*}
\left. \times \left( v_{mt}\left( \tau ,x\right) -v_{lt}\left( \tau
,x\right) \right) dxd\tau \right\vert
\end{equation*}%
\begin{equation*}
\leq \left\vert \int\limits_{0}^{t}\int\limits_{B\left( 0,r\right) }\tau
\left( f_{\varepsilon }\left( \left\Vert v_{m}\left( \tau \right)
\right\Vert _{L^{p}\left(
\mathbb{R}
^{n}\right) }\right) \left\vert v_{m}\left( \tau ,x\right) \right\vert
^{p-2}v_{m}(\tau ,x)-f_{\varepsilon }\left( \left\Vert v_{l}\left( \tau
\right) \right\Vert _{L^{p}\left(
\mathbb{R}
^{n}\right) }\right) \left\vert v_{l}(\tau ,x)\right\vert ^{p-2}v_{l}(\tau
,x)\right) \right.
\end{equation*}%
\begin{equation}
\left. \times \left( v_{mt}\left( \tau ,x\right) -v_{lt}\left( \tau
,x\right) \right) dxd\tau \right\vert +c\text{ }t\max_{0\leq s_{1},s_{2}\leq
\varepsilon }\left\vert f\left( s_{1}\right) -f\left( s_{2}\right)
\right\vert ,\text{ \ }\forall t\geq 0\text{.}  \tag{3.1}
\end{equation}%
Let us estimate the first term on the right hand side of (3.1).%
\begin{equation*}
\int\limits_{0}^{t}\int\limits_{B\left( 0,r\right) }\tau \left(
f_{\varepsilon }\left( \left\Vert v_{m}\left( \tau \right) \right\Vert
_{L^{p}\left(
\mathbb{R}
^{n}\right) }\right) \left\vert v_{m}\left( \tau ,x\right) \right\vert
^{p-2}v_{m}(\tau ,x)-f_{\varepsilon }\left( \left\Vert v_{l}\left( \tau
\right) \right\Vert _{L^{p}\left(
\mathbb{R}
^{n}\right) }\right) \left\vert v_{l}(\tau ,x)\right\vert ^{p-2}v_{l}(\tau
,x)\right)
\end{equation*}%
\begin{equation*}
\times \left( v_{mt}\left( \tau ,x\right) -v_{lt}\left( \tau ,x\right)
\right) dxd\tau
\end{equation*}%
\begin{equation*}
=\int\limits_{0}^{t}\tau f_{\varepsilon }\left( \left\Vert v_{m}\left( \tau
\right) \right\Vert _{L^{p}\left(
\mathbb{R}
^{n}\right) }\right) \int\limits_{B\left( 0,r\right) }\left\vert v_{m}\left(
\tau ,x\right) \right\vert ^{p-2}v_{m}(\tau ,x)v_{mt}\left( \tau ,x\right)
dxd\tau
\end{equation*}%
\begin{equation*}
+\int\limits_{0}^{t}\tau f_{\varepsilon }\left( \left\Vert v_{l}\left( \tau
\right) \right\Vert _{L^{p}\left(
\mathbb{R}
^{n}\right) }\right) \int\limits_{B\left( 0,r\right) }\left\vert v_{l}(\tau
,x)\right\vert ^{p-2}v_{l}(\tau ,x)v_{lt}\left( \tau ,x\right) dxd\tau
\end{equation*}%
\begin{equation*}
-\int\limits_{0}^{t}\tau f_{\varepsilon }\left( \left\Vert v_{m}\left( \tau
\right) \right\Vert _{L^{p}\left(
\mathbb{R}
^{n}\right) }\right) \int\limits_{B\left( 0,r\right) }\left\vert v_{m}\left(
\tau ,x\right) \right\vert ^{p-2}v_{m}(\tau ,x)v_{lt}\left( \tau ,x\right)
dxd\tau
\end{equation*}%
\begin{equation}
-\int\limits_{0}^{t}\tau f_{\varepsilon }\left( \left\Vert v_{l}\left( \tau
\right) \right\Vert _{L^{p}\left(
\mathbb{R}
^{n}\right) }\right) \int\limits_{B\left( 0,r\right) \ }\left\vert
v_{l}(\tau ,x)\right\vert ^{p-2}v_{l}(\tau ,x)v_{mt}\left( \tau ,x\right)
dxd\tau .  \tag{3.2}
\end{equation}%
For the first two terms on the right hand side of (3.2), we have%
\begin{equation*}
\int\limits_{0}^{t}\tau f_{\varepsilon }\left( \left\Vert v_{m}\left( \tau
\right) \right\Vert _{L^{p}\left(
\mathbb{R}
^{n}\right) }\right) \int\limits_{B\left( 0,r\right) }\left\vert v_{m}\left(
\tau ,x\right) \right\vert ^{p-2}v_{m}(\tau ,x)v_{mt}\left( \tau ,x\right)
dxd\tau
\end{equation*}%
\begin{equation*}
+\int\limits_{0}^{t}\tau f_{\varepsilon }\left( \left\Vert v_{l}\left( \tau
\right) \right\Vert _{L^{p}\left(
\mathbb{R}
^{n}\right) }\right) \int\limits_{B\left( 0,r\right) }\left\vert v_{l}(\tau
,x)\right\vert ^{p-2}v_{l}(\tau ,x)v_{lt}\left( \tau ,x\right) dxd\tau
\end{equation*}%
\begin{equation*}
=\frac{1}{p}tf_{\varepsilon }\left( \left\Vert v_{m}\left( t\right)
\right\Vert _{L^{p}\left(
\mathbb{R}
^{n}\right) }\right) \left\Vert v_{m}\left( t\right) \right\Vert
_{L^{p}\left( B\left( 0,r\right) \right) }^{p}+\frac{1}{p}tf_{\varepsilon
}\left( \left\Vert v_{l}\left( t\right) \right\Vert _{L^{p}\left(
\mathbb{R}
^{n}\right) }\right) \left\Vert v_{l}\left( t\right) \right\Vert
_{L^{p}\left( B\left( 0,r\right) \right) }^{p}
\end{equation*}%
\begin{equation*}
-\frac{1}{p}\int\limits_{0}^{t}f_{\varepsilon }\left( \left\Vert v_{m}\left(
\tau \right) \right\Vert _{L^{p}\left(
\mathbb{R}
^{n}\right) }\right) \left\Vert v_{m}\left( \tau \right) \right\Vert
_{L^{p}\left( B\left( 0,r\right) \right) }^{p}d\tau -\frac{1}{p}%
\int\limits_{0}^{t}f_{\varepsilon }\left( \left\Vert v_{l}\left( \tau
\right) \right\Vert _{L^{p}\left(
\mathbb{R}
^{n}\right) }\right) \left\Vert v_{l}\left( \tau \right) \right\Vert
_{L^{p}\left( B\left( 0,r\right) \right) }^{p}d\tau
\end{equation*}%
\begin{equation*}
-\frac{1}{p}\int\limits_{0}^{t}\tau \frac{d}{dt}\left( f_{\varepsilon
}\left( \left\Vert v_{m}\left( \tau \right) \right\Vert _{L^{p}\left(
\mathbb{R}
^{n}\right) }\right) \right) \left\Vert v_{m}\left( \tau \right) \right\Vert
_{L^{p}\left( B\left( 0,r\right) \right) }^{p}d\tau
\end{equation*}%
\begin{equation*}
-\frac{1}{p}\int\limits_{0}^{t}\tau \frac{d}{dt}\left( f_{\varepsilon
}\left( \left\Vert v_{l}\left( \tau \right) \right\Vert _{L^{p}\left(
\mathbb{R}
^{n}\right) }\right) \right) \left\Vert v_{l}\left( \tau \right) \right\Vert
_{L^{p}\left( B\left( 0,r\right) \right) }^{p}d\tau .
\end{equation*}%
Since the sequence $\left\{ \left\Vert v_{m}\left( t\right) \right\Vert
_{L^{p}\left(
\mathbb{R}
^{n}\right) }\right\} _{m=1}^{\infty }$ is convergent, by continuity of $%
f_{\varepsilon },$ it follows that the sequence $\left\{ f_{\varepsilon
}\left( \left\Vert v_{m}\left( t\right) \right\Vert _{L^{p}\left(
\mathbb{R}
^{n}\right) }\right) \right\} _{m=1}^{\infty }$ also converges for all $t\in %
\left[ 0,\infty \right) $. Moreover, by the conditions of the lemma and the
definition of $f_{\varepsilon }$,\ we obtain that the sequence $\left\{
f_{\varepsilon }\left( \left\Vert v_{m}\left( .\right) \right\Vert
_{L^{p}\left(
\mathbb{R}
^{n}\right) }\right) \right\} _{m=1}^{\infty }$ is bounded in $W^{1,\infty
}\left( 0,\infty \right) $. So, the sequence $\left\{ f_{\varepsilon }\left(
\left\Vert v_{m}\left( .\right) \right\Vert _{L^{p}\left(
\mathbb{R}
^{n}\right) }\right) \right\} _{m=1}^{\infty }$ converges weakly star in $%
W^{1,\infty }\left( 0,\infty \right) $ and we have%
\begin{equation}
\left\{
\begin{array}{c}
f_{\varepsilon }\left( \left\Vert v_{m}\left( .\right) \right\Vert
_{L^{p}\left(
\mathbb{R}
^{n}\right) }\right) \rightarrow Q\text{ \ \ weakly star in \ }W^{1,\infty
}\left( 0,\infty \right) , \\
v_{m}\rightarrow v\text{ \ \ weakly star in }L^{\infty }\left( 0,\infty
;H^{2}\left(
\mathbb{R}
^{n}\right) \right) , \\
v_{mt}\rightarrow v_{t}\text{ \ \ weakly star in \ \ }L^{\infty }\left(
0,\infty ;L^{2}\left(
\mathbb{R}
^{n}\right) \right) ,%
\end{array}%
\right.  \tag{3.3}
\end{equation}%
for some $Q\in W^{1,\infty }\left( 0,\infty \right) $ and $v\in L^{\infty
}\left( 0,\infty ;H^{2}\left(
\mathbb{R}
^{n}\right) \right) \cap W^{1,\infty }\left( 0,\infty ;L^{2}\left(
\mathbb{R}
^{n}\right) \right) $. Applying Aubin--Lions--Simon lemma (see [20]), by
(3.3)$_{2}$-(3.3)$_{3}$, we find%
\begin{equation}
v_{m}\rightarrow v\text{ strongly in }C\left( \left[ 0,T\right] ;L^{q}\left(
B\left( 0,r\right) \right) \right) ,\text{ \ }\forall T\geq 0,  \tag{3.4}
\end{equation}%
where $q<\frac{2n}{\left( n-4\right) ^{+}}$. Then, considering (3.3) and
(3.4), we get%
\begin{equation*}
\lim_{m\rightarrow \infty }\int\limits_{0}^{t}\tau f_{\varepsilon }\left(
\left\Vert v_{m}\left( \tau \right) \right\Vert _{L^{p}\left(
\mathbb{R}
^{n}\right) }\right) \int\limits_{B\left( 0,r\right) }\left\vert v_{m}\left(
\tau ,x\right) \right\vert ^{p-2}v_{m}(\tau ,x)v_{mt}\left( \tau ,x\right)
dxd\tau
\end{equation*}%
\begin{equation*}
+\lim_{l\rightarrow \infty }\int\limits_{0}^{t}\tau f_{\varepsilon }\left(
\left\Vert v_{l}\left( \tau \right) \right\Vert _{L^{p}\left(
\mathbb{R}
^{n}\right) }\right) \int\limits_{B\left( 0,r\right) }\left\vert v_{l}(\tau
,x)\right\vert ^{p-2}v_{l}(\tau ,x)v_{lt}\left( \tau ,x\right) dxd\tau
\end{equation*}%
\begin{equation*}
=\frac{2}{p}tQ\left( t\right) \left\Vert v\left( t\right) \right\Vert
_{L^{p}\left( B\left( 0,r\right) \right) }^{p}-\frac{2}{p}%
\int\limits_{0}^{t}Q\left( \tau \right) \left\Vert v\left( \tau \right)
\right\Vert _{L^{p}\left( B\left( 0,r\right) \right) }^{p}d\tau
\end{equation*}%
\begin{equation}
-\frac{2}{p}\int\limits_{0}^{t}\tau \frac{d}{dt}\left( Q\left( \tau \right)
\right) \left\Vert v\left( \tau \right) \right\Vert _{L^{p}\left( B\left(
0,r\right) \right) }^{p}d\tau \text{.}  \tag{3.5}
\end{equation}%
For the last two terms on the right hand side of (3.2), by using (3.3), we
have%
\begin{equation*}
\lim_{m\rightarrow \infty }\lim_{l\rightarrow \infty
}\int\limits_{0}^{t}\tau f_{\varepsilon }\left( \left\Vert v_{m}\left( \tau
\right) \right\Vert _{L^{p}\left(
\mathbb{R}
^{n}\right) }\right) \int\limits_{B\left( 0,r\right) }\left\vert v_{m}\left(
\tau ,x\right) \right\vert ^{p-2}v_{m}(\tau ,x)v_{lt}\left( \tau ,x\right)
dxd\tau
\end{equation*}%
\begin{equation*}
+\lim_{m\rightarrow \infty }\lim_{l\rightarrow \infty
}\int\limits_{0}^{t}\tau f_{\varepsilon }\left( \left\Vert v_{l}\left( \tau
\right) \right\Vert _{L^{p}\left(
\mathbb{R}
^{n}\right) }\right) \int\limits_{B\left( 0,r\right) \ }\left\vert
v_{l}(\tau ,x)\right\vert ^{p-2}v_{l}(\tau ,x)v_{mt}\left( \tau ,x\right)
dxd\tau
\end{equation*}%
\begin{equation*}
=2\int\limits_{0}^{t}\tau Q\left( \tau \right) \int\limits_{B\left(
0,r\right) \ }\left\vert v(\tau ,x)\right\vert ^{p-2}v(\tau ,x)v_{t}\left(
\tau ,x\right) dxd\tau
\end{equation*}%
\begin{equation*}
=\frac{2}{p}tQ\left( t\right) \left\Vert v\left( t\right) \right\Vert
_{L^{p}\left( B\left( 0,r\right) \right) }^{p}-\frac{2}{p}%
\int\limits_{0}^{t}Q\left( \tau \right) \left\Vert v\left( \tau \right)
\right\Vert _{L^{p}\left( B\left( 0,r\right) \right) }^{p}d\tau
\end{equation*}%
\begin{equation}
-\frac{2}{p}\int\limits_{0}^{t}\tau \frac{d}{dt}\left( Q\left( \tau \right)
\right) \left\Vert v\left( \tau \right) \right\Vert _{L^{p}\left( B\left(
0,r\right) \right) }^{p}d\tau \text{.}  \tag{3.6}
\end{equation}%
Hence, taking into account (3.5)-(3.6) and passing to limit in (3.2), we
obtain%
\begin{equation*}
\underset{m\rightarrow \infty }{\lim }\underset{l\rightarrow \infty }{\lim }%
\int\limits_{0}^{t}\int\limits_{B\left( 0,r\right) }\tau f_{\varepsilon
}\left( \left\Vert v_{m}\left( \tau \right) \right\Vert _{L^{p}\left(
\mathbb{R}
^{n}\right) }\right) \left\vert v_{m}\left( \tau ,x\right) \right\vert
^{p-2}v_{m}(\tau ,x)
\end{equation*}%
\begin{equation}
-f_{\varepsilon }\left( \left\Vert v_{l}\left( \tau \right) \right\Vert
_{L^{p}\left(
\mathbb{R}
^{n}\right) }\right) \left\vert v_{l}(\tau ,x)\right\vert ^{p-2}v_{l}(\tau
,x)\left( v_{mt}\left( \tau ,x\right) -v_{lt}\left( \tau ,x\right) \right)
dxd\tau =0\text{.}  \tag{3.7}
\end{equation}%
Then, by (3.1) and (3.7), for all $r>0$, we have
\begin{equation*}
\underset{m\rightarrow \infty }{\lim \sup }\text{ }\underset{l\rightarrow
\infty }{\lim \sup }\left\vert \int\limits_{0}^{t}\int\limits_{B\left(
0,r\right) }\tau f\left( \left\Vert v_{m}\left( \tau \right) \right\Vert
_{L^{p}\left(
\mathbb{R}
^{n}\right) }\right) \left\vert v_{m}\left( \tau ,x\right) \right\vert
^{p-2}v_{m}(\tau ,x)\right.
\end{equation*}%
\begin{equation*}
\left. -f\left( \left\Vert v_{l}\left( \tau \right) \right\Vert
_{L^{p}\left(
\mathbb{R}
^{n}\right) }\right) \left\vert v_{l}(\tau ,x)\right\vert ^{p-2}v_{l}(\tau
,x)\left( v_{mt}\left( \tau ,x\right) -v_{lt}\left( \tau ,x\right) \right)
dxd\tau \right\vert
\end{equation*}%
\begin{equation*}
\leq c\text{ }t\max_{0\leq s_{1},s_{2}\leq \varepsilon }\left\vert f\left(
s_{1}\right) -f\left( s_{2}\right) \right\vert ,\text{ }\forall t\geq 0\text{%
.}
\end{equation*}%
Thus, passing to the limit in the above inequality as $\varepsilon
\rightarrow 0$, we obtain the claim of the lemma.
\end{proof}

\begin{lemma}
Assume that in addition to the conditions of Lemma 3.1, conditions (2.3) and
(2.4) also hold. Then, for every $\gamma >0$ there exists $c_{\gamma }>0$
such that%
\begin{equation*}
\int\limits_{0}^{t}\int\limits_{%
\mathbb{R}
^{n}}\tau \left( f\left( \left\Vert v_{l}\left( \tau \right) \right\Vert
_{L^{p}\left(
\mathbb{R}
^{n}\right) }\right) \left\vert v_{l}(\tau ,x)\right\vert ^{p-2}v_{l}(\tau
,x)-f\left( \left\Vert v_{m}\left( \tau \right) \right\Vert _{L^{p}\left(
\mathbb{R}
^{n}\right) }\right) \left\vert v_{m}\left( \tau ,x\right) \right\vert
^{p-2}v_{m}(\tau ,x)\right)
\end{equation*}%
\begin{equation*}
\times \left( v_{mt}\left( \tau ,x\right) -v_{lt}\left( \tau ,x\right)
\right) dxd\tau
\end{equation*}%
\begin{equation*}
\leq \gamma \int\limits_{0}^{t}\tau E_{%
\mathbb{R}
^{n}\backslash \left( B\left( 0,r\right) \right) }\left( v_{m}\left( \tau
\right) -v_{l}\left( \tau \right) \right) d\tau +c_{\gamma
}\int\limits_{0}^{t}E_{%
\mathbb{R}
^{n}\backslash \left( B\left( 0,r\right) \right) }\left( v_{m}\left( \tau
\right) -v_{l}\left( \tau \right) \right) d\tau
\end{equation*}%
\begin{equation*}
+c_{\gamma }\int\limits_{0}^{t}\tau \left( \left\Vert \sqrt{a}v_{mt}\left(
\tau \right) \right\Vert _{L^{2}\left(
\mathbb{R}
^{n}\right) }^{2}+\left\Vert \sqrt{a}v_{lt}\left( \tau \right) \right\Vert
_{L^{2}\left(
\mathbb{R}
^{n}\right) }^{2}\right) E_{%
\mathbb{R}
^{n}\backslash B\left( 0,r\right) }\left( v_{m}\left( \tau \right)
-v_{l}\left( \tau \right) \right) d\tau
\end{equation*}%
\begin{equation*}
+K_{r}^{m,l}(t)\text{, }\forall t\geq 0\text{, }\forall r\geq r_{0},
\end{equation*}%
where $K_{r}^{m,l}\in C\left[ 0,t\right] $, $\underset{m,l}{\sup }\left\Vert
K_{r}^{m,l}\right\Vert _{C\left[ 0,t\right] }$ $<\infty $ \ and $\underset{%
m\rightarrow \infty }{\lim }$ $\underset{l\rightarrow \infty }{\lim \sup }%
\left\vert K_{r}^{m,l}\left( t\right) \right\vert =0$, for all $t$ $\geq 0$.
\end{lemma}

\begin{proof}
Firstly, we have%
\begin{equation*}
\int\limits_{0}^{t}\int\limits_{%
\mathbb{R}
^{n}}\tau \left( f\left( \left\Vert v_{l}\left( \tau \right) \right\Vert
_{L^{p}\left(
\mathbb{R}
^{n}\right) }\right) \left\vert v_{l}(\tau ,x)\right\vert ^{p-2}v_{l}(\tau
,x)-f\left( \left\Vert v_{m}\left( \tau \right) \right\Vert _{L^{p}\left(
\mathbb{R}
^{n}\right) }\right) \left\vert v_{m}\left( \tau ,x\right) \right\vert
^{p-2}v_{m}(\tau ,x)\right)
\end{equation*}%
\begin{equation*}
\times \left( v_{mt}\left( \tau ,x\right) -v_{lt}\left( \tau ,x\right)
\right) dxd\tau
\end{equation*}%
\begin{equation*}
=\int\limits_{0}^{t}\int\limits_{%
\mathbb{R}
^{n}\backslash B\left( 0,r\right) }\tau \left( f\left( \left\Vert
v_{l}\left( \tau \right) \right\Vert _{L^{p}\left(
\mathbb{R}
^{n}\right) }\right) \left\vert v_{l}(\tau ,x)\right\vert ^{p-2}v_{l}(\tau
,x)-f\left( \left\Vert v_{m}\left( \tau \right) \right\Vert _{L^{p}\left(
\mathbb{R}
^{n}\right) }\right) \left\vert v_{m}\left( \tau ,x\right) \right\vert
^{p-2}v_{m}(\tau ,x)\right)
\end{equation*}%
\begin{equation}
\times \left( v_{mt}\left( \tau ,x\right) -v_{lt}\left( \tau ,x\right)
\right) dxd\tau +K_{1,r}^{m,l}(t)\text{, }\forall r>0,  \tag{3.8}
\end{equation}%
where%
\begin{equation*}
K_{1,r}^{m,l}(t):=\int\limits_{0}^{t}\int\limits_{B\left( 0,r\right) }\tau
\left( f\left( \left\Vert v_{l}\left( \tau \right) \right\Vert _{L^{p}\left(
\mathbb{R}
^{n}\right) }\right) \left\vert v_{l}(\tau ,x)\right\vert ^{p-2}v_{l}(\tau
,x)\right.
\end{equation*}%
\begin{equation*}
\left. -f\left( \left\Vert v_{m}\left( \tau \right) \right\Vert
_{L^{p}\left(
\mathbb{R}
^{n}\right) }\right) \left\vert v_{m}\left( \tau ,x\right) \right\vert
^{p-2}v_{m}(\tau ,x)\right) \left( v_{mt}\left( \tau ,x\right) -v_{lt}\left(
\tau ,x\right) \right) dxd\tau \text{,}
\end{equation*}%
and by Lemma 3.1, it follows that%
\begin{equation*}
\underset{m,l}{\sup }\left\Vert K_{1,r}^{m,l}\right\Vert _{C\left[ 0,t\right]
}<\infty \text{ \ and }\underset{m\rightarrow \infty }{\text{ }\lim \text{ }}%
\underset{l\rightarrow \infty }{\lim \sup }\left\vert
K_{1,r}^{m,l}(t)\right\vert =0,\text{ }\forall t\geq 0.
\end{equation*}%
On the other hand, for the first term on the right hand side of (3.8), we get%
\begin{equation*}
\int\limits_{0}^{t}\int\limits_{%
\mathbb{R}
^{n}\backslash B\left( 0,r\right) }\tau \left( f\left( \left\Vert
v_{l}\left( \tau \right) \right\Vert _{L^{p}\left(
\mathbb{R}
^{n}\right) }\right) \left\vert v_{l}(\tau ,x)\right\vert ^{p-2}v_{l}(\tau
,x)-f\left( \left\Vert v_{m}\left( \tau \right) \right\Vert _{L^{p}\left(
\mathbb{R}
^{n}\right) }\right) \left\vert v_{m}\left( \tau ,x\right) \right\vert
^{p-2}v_{m}(\tau ,x)\right)
\end{equation*}%
\begin{equation*}
\times \left( v_{mt}\left( \tau ,x\right) -v_{lt}\left( \tau ,x\right)
\right) dxd\tau
\end{equation*}%
\begin{equation*}
=-\frac{\left( p-1\right) }{2}\int\limits_{0}^{t}\tau f\left( \left\Vert
v_{m}\left( \tau \right) \right\Vert _{L^{p}\left(
\mathbb{R}
^{n}\right) }\right) \int\limits_{%
\mathbb{R}
^{n}\backslash B\left( 0,r\right) }\int_{0}^{1}\left\vert v_{m}\left( \tau
,x\right) +\sigma \left( v_{l}\left( \tau ,x\right) -v_{m}\left( \tau
,x\right) \right) \right\vert ^{p-2}d\sigma
\end{equation*}%
\begin{equation}
\times \frac{d}{d\tau }\left\vert v_{m}\left( \tau ,x\right) -v_{l}\left(
\tau ,x\right) \right\vert ^{2}dxd\tau +K_{2,r}^{m,l}\left( t\right) ,
\tag{3.9}
\end{equation}%
where%
\begin{equation*}
K_{2,r}^{m,l}(t):=\int\limits_{0}^{t}\tau \left( f\left( \left\Vert
v_{l}\left( \tau \right) \right\Vert _{L^{p}\left(
\mathbb{R}
^{n}\right) }\right) -f\left( \left\Vert v_{m}\left( \tau \right)
\right\Vert _{L^{p}\left(
\mathbb{R}
^{n}\right) }\right) \right) \int\limits_{%
\mathbb{R}
^{n}\backslash B\left( 0,r\right) }\left\vert v_{l}(\tau ,x)\right\vert
^{p-2}v_{l}(\tau ,x)
\end{equation*}%
\begin{equation*}
\times \left( v_{mt}\left( \tau ,x\right) -v_{lt}\left( \tau ,x\right)
\right) dxd\tau \text{.}
\end{equation*}%
By the conditions of the lemma, we find
\begin{equation*}
\underset{m,l}{\sup }\left\Vert K_{2,r}^{m,l}\right\Vert _{C\left[ 0,t\right]
}<\infty \text{ \ and}\underset{m\rightarrow \infty }{\text{ }\lim \text{ }}%
\underset{l\rightarrow \infty }{\lim \sup }\left\vert K_{2,r}^{m,l}\left(
t\right) \right\vert =0\text{, }\forall t\geq 0\text{.}
\end{equation*}%
So, denoting $K_{r}^{m,l}\left( t\right) :=K_{1,r}^{m,l}\left( t\right)
+K_{2,r}^{m,l}\left( t\right) $, by (3.8) and (3.9), we have%
\begin{equation*}
\int\limits_{0}^{t}\int\limits_{%
\mathbb{R}
^{n}}\tau \left( f\left( \left\Vert v_{l}\left( \tau \right) \right\Vert
_{L^{p}\left(
\mathbb{R}
^{n}\right) }\right) \left\vert v_{l}(\tau ,x)\right\vert ^{p-2}v_{l}(\tau
,x)-f\left( \left\Vert v_{m}\left( \tau \right) \right\Vert _{L^{p}\left(
\mathbb{R}
^{n}\right) }\right) \left\vert v_{m}\left( \tau ,x\right) \right\vert
^{p-2}v_{m}(\tau ,x)\right)
\end{equation*}%
\begin{equation*}
\times \left( v_{mt}\left( \tau ,x\right) -v_{lt}\left( \tau ,x\right)
\right) dxd\tau
\end{equation*}%
\begin{equation*}
=-\frac{\left( p-1\right) }{2}\int\limits_{0}^{t}\tau f\left( \left\Vert
v_{m}\left( \tau \right) \right\Vert _{L^{p}\left(
\mathbb{R}
^{n}\right) }\right) \int\limits_{%
\mathbb{R}
^{n}\backslash B\left( 0,r\right) }\int_{0}^{1}\left\vert v_{m}\left( \tau
,x\right) +\sigma \left( v_{l}\left( \tau ,x\right) -v_{m}\left( \tau
,x\right) \right) \right\vert ^{p-2}d\sigma
\end{equation*}%
\begin{equation}
\times \frac{d}{d\tau }\left\vert v_{m}\left( \tau ,x\right) -v_{l}\left(
\tau ,x\right) \right\vert ^{2}dxd\tau +K_{r}^{m,l}\left( t\right) .
\tag{3.10}
\end{equation}

Now, let us estimate the first term on the right hand side of (3.10). Denote
$\varphi _{M}\left( u\right) =$\newline
$\left\{
\begin{array}{c}
u,\text{ }\left\vert u\right\vert \leq M, \\
M,\text{ }\left\vert u\right\vert >M,%
\end{array}%
\right. $ and $\Psi _{\varepsilon }\left( u\right) =\left\{
\begin{array}{c}
\varepsilon ^{p-2},\text{ }\left\vert u\right\vert \leq \varepsilon , \\
\left\vert u\right\vert ^{p-2},\text{ }\left\vert u\right\vert >\varepsilon%
\end{array}%
\right. $. Then, we get%
\begin{equation*}
\left\vert \left\Vert v_{m}\left( \tau \right) \right\Vert _{L^{p}\left(
\mathbb{R}
^{n}\right) }-\left\Vert \varphi _{M}\left( v_{m}\left( \tau \right) \right)
\right\Vert _{L^{p}\left(
\mathbb{R}
^{n}\right) }\right\vert \leq 2\left( \int\limits_{\left\{ x\in
\mathbb{R}
^{n}:\left\vert v_{m}\left( \tau ,x\right) \right\vert >M\right\}
}\left\vert v_{m}\left( \tau ,x\right) \right\vert ^{^{p}}dx\right) ^{\frac{1%
}{p}}
\end{equation*}%
\begin{equation}
\leq \frac{2}{M^{\beta }}\left( \int\limits_{\left\{ x\in
\mathbb{R}
^{n}:\left\vert v_{m}\left( \tau ,x\right) \right\vert >M\right\}
}\left\vert v_{m}\left( \tau ,x\right) \right\vert ^{^{p+\beta }}dx\right) ^{%
\frac{1}{p}}\leq \frac{2}{M^{\beta }}\left\Vert v\left( \tau \right)
\right\Vert _{H^{2}\left(
\mathbb{R}
^{n}\right) }^{\frac{p+\beta }{p}},  \tag{3.11}
\end{equation}%
where $\beta \in \left( 0,\frac{2n}{\left( n-4\right) ^{+}}-p\right) $.
Also, it is clear that%
\begin{equation}
\left\vert \left\vert w\right\vert ^{p-2}-\Psi _{\varepsilon }\left(
w\right) \right\vert \leq \omega \left( \varepsilon \right) \text{,}
\tag{3.12}
\end{equation}%
where $\omega \left( \varepsilon \right) =\left\{
\begin{array}{c}
\varepsilon ^{p-2}\text{, \ \ }p>2, \\
0\text{, \ \ \ \ }p=2%
\end{array}%
\right. $.

By (3.11) and (3.12), it is easy to see that%
\begin{equation*}
\int\limits_{0}^{t}\tau f\left( \left\Vert v_{m}\left( \tau \right)
\right\Vert _{L^{p}\left(
\mathbb{R}
^{n}\right) }\right) \int\limits_{%
\mathbb{R}
^{n}\backslash B\left( 0,r\right) }\int_{0}^{1}\left\vert v_{m}\left( \tau
,x\right) +\sigma \left( v_{l}\left( \tau ,x\right) -v_{m}\left( \tau
,x\right) \right) \right\vert ^{p-2}d\sigma
\end{equation*}%
\begin{equation*}
\times \frac{d}{d\tau }\left\vert v_{m}\left( \tau ,x\right) -v_{l}\left(
\tau ,x\right) \right\vert ^{2}dxd\tau
\end{equation*}%
\begin{equation*}
\geq \int\limits_{0}^{t}\tau f_{\varepsilon }\left( \left\Vert \varphi
_{M}\left( v_{m}\left( \tau \right) \right) \right\Vert _{L^{p}\left(
\mathbb{R}
^{n}\right) }\right) \int\limits_{%
\mathbb{R}
^{n}\backslash B\left( 0,r\right) }\int_{0}^{1}\Psi _{\varepsilon }\left(
v_{m}\left( \tau ,x\right) +\sigma \left( v_{l}\left( \tau ,x\right)
-v_{m}\left( \tau ,x\right) \right) \right) d\sigma
\end{equation*}%
\begin{equation*}
\times \frac{d}{d\tau }\left\vert v_{m}\left( \tau ,x\right) -v_{l}\left(
\tau ,x\right) \right\vert ^{2}dxd\tau
\end{equation*}%
\begin{equation}
-c_{1}\left( \max_{0<s_{1},s_{2}<\varepsilon }\left\vert f\left(
s_{1}\right) -f\left( s_{2}\right) \right\vert +\frac{1}{M^{\beta }}+\omega
\left( \varepsilon \right) \right) \int\limits_{0}^{t}\tau E_{%
\mathbb{R}
^{n}\backslash B\left( 0,r\right) }\left( v_{m}\left( \tau \right)
-v_{l}\left( \tau \right) \right) d\tau ,  \tag{3.13}
\end{equation}%
where $f_{\varepsilon }\left( \cdot \right) $ is as in Lemma 3.1.

Now, let us estimate the first term on the right hand side of (3.13).%
\begin{equation*}
\int\limits_{0}^{t}\tau f_{\varepsilon }\left( \left\Vert \varphi _{M}\left(
v_{m}\left( \tau \right) \right) \right\Vert _{L^{p}\left(
\mathbb{R}
^{n}\right) }\right) \int\limits_{%
\mathbb{R}
^{n}\backslash B\left( 0,r\right) }\int_{0}^{1}\Psi _{\varepsilon }\left(
v_{m}\left( \tau ,x\right) +\sigma \left( v_{l}\left( \tau ,x\right)
-v_{m}\left( \tau ,x\right) \right) \right) d\sigma
\end{equation*}%
\begin{equation*}
\times \frac{d}{d\tau }\left\vert v_{m}\left( \tau ,x\right) -v_{l}\left(
\tau ,x\right) \right\vert ^{2}dxd\tau
\end{equation*}%
\begin{equation*}
\geq -\int\limits_{0}^{t}f_{\varepsilon }\left( \left\Vert \varphi
_{M}\left( v_{m}\left( \tau \right) \right) \right\Vert _{L^{p}\left(
\mathbb{R}
^{n}\right) }\right) \int\limits_{%
\mathbb{R}
^{n}\backslash B\left( 0,r\right) }\int_{0}^{1}\Psi _{\varepsilon }\left(
v_{m}\left( \tau ,x\right) +\sigma \left( v_{l}\left( \tau ,x\right)
-v_{m}\left( \tau ,x\right) \right) \right) d\sigma
\end{equation*}%
\begin{equation*}
\times \left\vert v_{m}\left( \tau ,x\right) -v_{l}\left( \tau ,x\right)
\right\vert ^{2}dxd\tau
\end{equation*}%
\begin{equation*}
-\int\limits_{0}^{t}\tau \frac{d}{d\tau }\left( f_{\varepsilon }\left(
\left\Vert \varphi _{M}\left( v_{m}\left( \tau \right) \right) \right\Vert
_{L^{p}\left(
\mathbb{R}
^{n}\right) }\right) \right) \int\limits_{%
\mathbb{R}
^{n}\backslash B\left( 0,r\right) }\int_{0}^{1}\Psi _{\varepsilon }\left(
v_{m}\left( \tau ,x\right) +\sigma \left( v_{l}\left( \tau ,x\right)
-v_{m}\left( \tau ,x\right) \right) \right) d\sigma
\end{equation*}%
\begin{equation*}
\times \left\vert v_{m}\left( \tau ,x\right) -v_{l}\left( \tau ,x\right)
\right\vert ^{2}dxd\tau
\end{equation*}%
\begin{equation*}
-\int\limits_{0}^{t}\tau f_{\varepsilon }\left( \left\Vert \varphi
_{M}\left( v_{m}\left( \tau \right) \right) \right\Vert _{L^{p}\left(
\mathbb{R}
^{n}\right) }\right) \int\limits_{%
\mathbb{R}
^{n}\backslash B\left( 0,r\right) }\int_{0}^{1}\frac{d}{d\tau }\left( \Psi
_{\varepsilon }\left( v_{m}\left( \tau ,x\right) +\sigma \left( v_{l}\left(
\tau ,x\right) -v_{m}\left( \tau ,x\right) \right) \right) \right) d\sigma
\end{equation*}%
\begin{equation}
\times \left\vert v_{m}\left( \tau ,x\right) -v_{l}\left( \tau ,x\right)
\right\vert ^{2}dxd\tau \text{.}  \tag{3.14}
\end{equation}%
Since
\begin{equation*}
\frac{d}{d\tau }\left( f_{\varepsilon }\left( \left\Vert \varphi _{M}\left(
v_{m}\left( \tau \right) \right) \right\Vert _{L^{p}\left(
\mathbb{R}
^{n}\right) }\right) \right)
\end{equation*}%
\begin{equation*}
=\frac{f_{\varepsilon }^{\prime }\left( \left\Vert \varphi _{M}\left(
v_{m}\left( \tau \right) \right) \right\Vert _{L^{p}\left(
\mathbb{R}
^{n}\right) }\right) }{\left\Vert \varphi _{M}\left( v_{m}\left( \tau
\right) \right) \right\Vert _{L^{p}\left(
\mathbb{R}
^{n}\right) }^{p-1}}\int\limits_{%
\mathbb{R}
^{n}}\left\vert \varphi _{M}\left( v_{m}\left( \tau ,x\right) \right)
\right\vert ^{p-1}\varphi _{M}^{\prime }\left( v_{m}\left( \tau ,x\right)
\right) v_{mt}\left( \tau ,x\right) dx
\end{equation*}%
\begin{equation*}
\leq \frac{c_{2}}{\varepsilon ^{p-1}}\left( \int\limits_{%
\mathbb{R}
^{n}\backslash B\left( 0,r_{0}\right) }\left\vert v_{m}\left( \tau ,x\right)
\right\vert ^{p-1}\left\vert v_{mt}\left( \tau ,x\right) \right\vert
dx+M^{p-1}\int\limits_{B\left( 0,r_{0}\right) }\left\vert v_{mt}\left( \tau
,x\right) \right\vert dx\right)
\end{equation*}%
\begin{equation*}
\leq \frac{c_{3}}{\varepsilon ^{p-1}}\left( \left\Vert v_{mt}\left( \tau
\right) \right\Vert _{L^{2}\left(
\mathbb{R}
^{n}\backslash B\left( 0,r_{0}\right) \right) }+M^{p-1}\left\Vert
v_{mt}\left( \tau \right) \right\Vert _{L^{1}\left( B\left( 0,r_{0}\right)
\right) }\right)
\end{equation*}%
and%
\begin{equation*}
\int\limits_{%
\mathbb{R}
^{n}\backslash B\left( 0,r\right) }\int_{0}^{1}\frac{d}{d\tau }\left( \Psi
_{\varepsilon }\left( v_{l}\left( \tau ,x\right) +\sigma \left( v_{m}\left(
\tau ,x\right) -v_{l}\left( \tau ,x\right) \right) \right) \right) d\sigma
\left\vert v_{m}\left( \tau ,x\right) -v_{l}\left( \tau ,x\right)
\right\vert ^{2}dx
\end{equation*}%
\begin{equation*}
\leq \int\limits_{%
\mathbb{R}
^{n}\backslash B\left( 0,r\right) }\int_{0}^{1}\Psi _{\varepsilon }^{\prime
}\left( v_{l}\left( \tau ,x\right) +\sigma \left( v_{m}\left( \tau ,x\right)
-v_{l}\left( \tau ,x\right) \right) \right) d\sigma \left\vert v_{m}\left(
\tau ,x\right) -v_{l}\left( \tau ,x\right) \right\vert ^{2}dx
\end{equation*}%
\begin{equation*}
\times \left( \left\Vert v_{mt}\left( \tau \right) \right\Vert _{L^{2}\left(
\mathbb{R}
^{n}\backslash B\left( 0,r\right) \right) }+\left\Vert v_{lt}\left( \tau
\right) \right\Vert _{L^{2}\left(
\mathbb{R}
^{n}\backslash B\left( 0,r\right) \right) }\right)
\end{equation*}%
\begin{equation*}
\leq \frac{c_{4}}{\varepsilon ^{\max \left\{ 0,3-p\right\} }}\left\Vert
v_{m}\left( \tau \right) -v_{l}\left( \tau \right) \right\Vert _{H^{2}\left(
\mathbb{R}
^{n}\backslash B\left( 0,r\right) \right) }^{2}
\end{equation*}%
\begin{equation*}
\times \left( \left\Vert v_{mt}\left( \tau \right) \right\Vert _{L^{2}\left(
\mathbb{R}
^{n}\backslash B\left( 0,r\right) \right) }+\left\Vert v_{lt}\left( \tau
\right) \right\Vert _{L^{2}\left(
\mathbb{R}
^{n}\backslash B\left( 0,r\right) \right) }\right)
\end{equation*}%
by (3.14), we find%
\begin{equation*}
\int\limits_{0}^{t}\tau f_{\varepsilon }\left( \left\Vert \varphi _{M}\left(
v_{m}\left( \tau \right) \right) \right\Vert _{L^{p}\left(
\mathbb{R}
^{n}\right) }\right) \int\limits_{%
\mathbb{R}
^{n}\backslash B\left( 0,r\right) }\int_{0}^{1}\Psi _{\varepsilon }\left(
v_{l}\left( \tau ,x\right) +\sigma \left( v_{m}\left( \tau ,x\right)
-v_{l}\left( \tau ,x\right) \right) \right) d\sigma
\end{equation*}%
\begin{equation*}
\times \frac{d}{d\tau }\left\vert v_{m}\left( \tau ,x\right) -v_{l}\left(
\tau ,x\right) \right\vert ^{2}dxd\tau
\end{equation*}%
\begin{equation*}
\geq -c_{5}\int\limits_{0}^{t}E_{%
\mathbb{R}
^{n}\backslash B\left( 0,r\right) }\left( v_{m}\left( \tau \right)
-v_{l}\left( \tau \right) \right) d\tau
\end{equation*}%
\begin{equation*}
-\frac{c_{5}}{\varepsilon ^{p-1}}\int\limits_{0}^{t}\tau E_{%
\mathbb{R}
^{n}\backslash B\left( 0,r\right) }\left( v_{m}\left( \tau \right)
-v_{l}\left( \tau \right) \right) \left\Vert v_{mt}\left( \tau \right)
\right\Vert _{L^{2}\left(
\mathbb{R}
^{n}\backslash B\left( 0,r_{0}\right) \right) }d\tau
\end{equation*}%
\begin{equation*}
-\frac{c_{5}M^{p-1}}{\varepsilon ^{p-1}}\int\limits_{0}^{t}\tau E_{%
\mathbb{R}
^{n}\backslash B\left( 0,r\right) }\left( v_{m}\left( \tau \right)
-v_{l}\left( \tau \right) \right) \left\Vert v_{mt}\left( \tau \right)
\right\Vert _{L^{1}\left( B\left( 0,r_{0}\right) \right) }d\tau
\end{equation*}%
\begin{equation*}
-\frac{c_{5}}{\varepsilon ^{\max \left\{ 0,3-p\right\} }}\int\limits_{0}^{t}%
\tau E_{%
\mathbb{R}
^{n}\backslash B\left( 0,r\right) }\left( v_{m}\left( \tau \right)
-v_{l}\left( \tau \right) \right) \left( \left\Vert v_{mt}\left( \tau
\right) \right\Vert _{L^{2}\left(
\mathbb{R}
^{n}\backslash B\left( 0,r\right) \right) }+\left\Vert v_{lt}\left( \tau
\right) \right\Vert _{L^{2}\left(
\mathbb{R}
^{n}\backslash B\left( 0,r\right) \right) }\right) d\tau \text{.}
\end{equation*}%
After applying Young inequality in the above inequality, we get%
\begin{equation*}
\int\limits_{0}^{t}\tau f_{\varepsilon }\left( \left\Vert \varphi _{M}\left(
v_{m}\left( \tau \right) \right) \right\Vert _{L^{p}\left(
\mathbb{R}
^{n}\right) }\right) \int\limits_{%
\mathbb{R}
^{n}\backslash B\left( 0,r\right) }\int_{0}^{1}\Psi _{\varepsilon }\left(
v_{l}\left( \tau ,x\right) +\sigma \left( v_{m}\left( \tau ,x\right)
-v_{l}\left( \tau ,x\right) \right) \right) d\sigma
\end{equation*}%
\begin{equation*}
\times \frac{d}{d\tau }\left\vert v_{m}\left( \tau ,x\right) -v_{l}\left(
\tau ,x\right) \right\vert ^{2}dxd\tau
\end{equation*}%
\begin{equation*}
\geq -c_{5}\int\limits_{0}^{t}E_{%
\mathbb{R}
^{n}\backslash B\left( 0,r\right) }\left( v_{m}\left( \tau \right)
-v_{l}\left( \tau \right) \right) d\tau -\frac{c_{5}}{\varepsilon ^{p-1}}\mu
\int\limits_{0}^{t}\tau E_{%
\mathbb{R}
^{n}\backslash B\left( 0,r\right) }\left( v_{m}\left( \tau \right)
-v_{l}\left( \tau \right) \right) d\tau
\end{equation*}%
\begin{equation*}
-\frac{c_{5}}{\varepsilon ^{p-1}\mu ^{2}}\int\limits_{0}^{t}\tau E_{%
\mathbb{R}
^{n}\backslash B\left( 0,r\right) }\left( v_{m}\left( \tau \right)
-v_{l}\left( \tau \right) \right) \left\Vert v_{mt}\left( \tau \right)
\right\Vert _{L^{2}\left(
\mathbb{R}
^{n}\backslash B\left( 0,r_{0}\right) \right) }^{2}d\tau
\end{equation*}%
\begin{equation*}
-\frac{c_{5}M^{p-1}}{\varepsilon ^{p-1}}\mu \int\limits_{0}^{t}\tau E_{%
\mathbb{R}
^{n}\backslash B\left( 0,r\right) }\left( v_{m}\left( \tau \right)
-v_{l}\left( \tau \right) \right) d\tau
\end{equation*}%
\begin{equation*}
-\frac{c_{5}M^{p-1}}{\varepsilon ^{p-1}\mu ^{2}}\int\limits_{0}^{t}\tau E_{%
\mathbb{R}
^{n}\backslash B\left( 0,r\right) }\left( v_{m}\left( \tau \right)
-v_{l}\left( \tau \right) \right) \left\Vert v_{mt}\left( \tau \right)
\right\Vert _{L^{1}\left( B\left( 0,r_{0}\right) \right) }^{2}d\tau
\end{equation*}%
\begin{equation*}
-\frac{c_{5}}{\varepsilon ^{\max \left\{ 0,3-p\right\} }}\mu
\int\limits_{0}^{t}\tau E_{%
\mathbb{R}
^{n}\backslash B\left( 0,r\right) }\left( v_{m}\left( \tau \right)
-v_{l}\left( \tau \right) \right) d\tau
\end{equation*}%
\begin{equation*}
-\frac{c_{5}}{\varepsilon ^{\max \left\{ 0,3-p\right\} }\mu ^{2}}%
\int\limits_{0}^{t}\tau E_{%
\mathbb{R}
^{n}\backslash B\left( 0,r\right) }\left( v_{m}\left( \tau \right)
-v_{l}\left( \tau \right) \right)
\end{equation*}%
\begin{equation}
\times \left( \left\Vert v_{mt}\left( \tau \right) \right\Vert _{L^{2}\left(
\mathbb{R}
^{n}\backslash B\left( 0,r\right) \right) }^{2}+\left\Vert v_{lt}\left( \tau
\right) \right\Vert _{L^{2}\left(
\mathbb{R}
^{n}\backslash B\left( 0,r\right) \right) }^{2}\right) d\tau \text{.}
\tag{3.15}
\end{equation}%
By (3.13) and (3.15), we obtain%
\begin{equation*}
\int\limits_{0}^{t}\tau f\left( \left\Vert v_{m}\left( \tau \right)
\right\Vert _{L^{p}\left(
\mathbb{R}
^{n}\right) }\right) \int\limits_{%
\mathbb{R}
^{n}\backslash B\left( 0,r\right) }\int_{0}^{1}\left\vert v_{m}\left( \tau
,x\right) +\sigma \left( v_{l}\left( \tau ,x\right) -v_{m}\left( \tau
,x\right) \right) \right\vert ^{p-2}d\sigma
\end{equation*}%
\begin{equation*}
\times \frac{d}{d\tau }\left\vert v_{m}\left( \tau ,x\right) -v_{l}\left(
\tau ,x\right) \right\vert ^{2}dxd\tau
\end{equation*}%
\begin{equation*}
\geq -c_{6}\int\limits_{0}^{t}E_{%
\mathbb{R}
^{n}\backslash B\left( 0,r\right) }\left( v_{m}\left( \tau \right)
-v_{l}\left( \tau \right) \right) d\tau
\end{equation*}%
\begin{equation*}
-c_{6}\left( \frac{\mu }{\varepsilon ^{p-1}}+\frac{\mu }{\varepsilon ^{\max
\left\{ 0,3-p\right\} }}+\max_{0\leq s_{1},s_{2}\leq \varepsilon }\left\vert
f\left( s_{1}\right) -f\left( s_{2}\right) \right\vert +\frac{M^{p-1}}{%
\varepsilon ^{p-1}}\mu \right.
\end{equation*}%
\begin{equation*}
\left. +\frac{1}{^{M^{\beta }}}+\omega \left( \varepsilon \right) \right)
\int\limits_{0}^{t}\tau E_{%
\mathbb{R}
^{n}\backslash B\left( 0,r\right) }\left( v_{m}\left( \tau \right)
-v_{l}\left( \tau \right) \right) d\tau
\end{equation*}%
\begin{equation*}
-c_{6}\left( \frac{1}{\varepsilon ^{p-1}\mu ^{2}}+\frac{1}{\varepsilon
^{\max \left\{ 0,3-p\right\} }\mu ^{2}}\right) \int\limits_{0}^{t}\tau E_{%
\mathbb{R}
^{n}\backslash B\left( 0,r\right) }\left( v_{m}\left( \tau \right)
-v_{l}\left( \tau \right) \right)
\end{equation*}%
\begin{equation*}
\times \left( \left\Vert v_{mt}\left( \tau \right) \right\Vert _{L^{2}\left(
\mathbb{R}
^{n}\backslash B\left( 0,r_{0}\right) \right) }^{2}+\left\Vert v_{lt}\left(
\tau \right) \right\Vert _{L^{2}\left(
\mathbb{R}
^{n}\backslash B\left( 0,r_{0}\right) \right) }^{2}\right) d\tau
\end{equation*}%
\begin{equation}
-c_{6}\frac{M^{p-1}}{\varepsilon ^{p-1}\mu ^{2}}\int\limits_{0}^{t}\tau E_{%
\mathbb{R}
^{n}\backslash B\left( 0,r\right) }\left( v_{m}\left( \tau \right)
-v_{l}\left( \tau \right) \right) \left\Vert v_{mt}\left( \tau \right)
\right\Vert _{L^{1}\left( B\left( 0,r_{0}\right) \right) }^{2}d\tau ,\text{
\ }\forall r\geq r_{0}.  \tag{3.16}
\end{equation}%
To complete the proof, let us estimate the term $\left\Vert v_{mt}\left(
\tau \right) \right\Vert _{L^{1}\left( B\left( 0,r_{0}\right) \right) }^{2}$%
. By the conditions of the lemma, we have%
\begin{equation*}
\left\Vert v_{mt}\left( \tau \right) \right\Vert _{L^{1}\left( B\left(
0,r_{0}\right) \right) }^{2}=\left( \int\limits_{B\left( 0,r_{0}\right)
}\left\vert v_{mt}\left( \tau ,x\right) \right\vert dx\right) ^{2}
\end{equation*}%
\begin{equation*}
=\left( \int\limits_{B\left( 0,r_{0}\right) }\frac{a\left( x\right) +\lambda
}{a\left( x\right) +\lambda }\left\vert v_{mt}\left( \tau ,x\right)
\right\vert dx\right) ^{2}
\end{equation*}%
\begin{equation*}
\leq \left( \frac{1}{\lambda }\int\limits_{B\left( 0,r_{0}\right) }a\left(
x\right) \left\vert v_{mt}\left( \tau ,x\right) \right\vert
dx+\int\limits_{B\left( 0,r_{0}\right) }\frac{\lambda }{a\left( x\right)
+\lambda }\left\vert v_{mt}\left( \tau ,x\right) \right\vert dx\right) ^{2}
\end{equation*}%
\begin{equation}
\leq \frac{c_{7}}{\lambda ^{2}}\int\limits_{%
\mathbb{R}
^{n}}a\left( x\right) \left\vert v_{mt}\left( \tau ,x\right) \right\vert
^{2}dx+c_{7}\int\limits_{B\left( 0,r_{0}\right) }\left( \frac{\lambda }{%
a\left( x\right) +\lambda }\right) ^{2}dx\text{, \ \ \ }\forall \lambda >0%
\text{.}  \tag{3.17}
\end{equation}%
Since, by Lebesgue dominated convergence theorem, $\lim_{\lambda \rightarrow
0^{+}}\int\limits_{B\left( 0,r_{0}\right) }\left( \frac{\lambda }{a\left(
x\right) +\lambda }\right) ^{2}dx=0$, we can choose positive parameters $%
\varepsilon ,$ $M,$ $\mu $ and $\lambda $ such that%
\begin{equation*}
c_{6}\left( \frac{\mu }{\varepsilon ^{p-1}}+\frac{\mu }{\varepsilon ^{\max
\left\{ 0,3-p\right\} }}+\max_{0\leq s_{1},s_{2}\leq \varepsilon }\left\vert
f\left( s_{1}\right) -f\left( s_{2}\right) \right\vert +\frac{M^{p-1}}{%
\varepsilon ^{p-1}}\mu \right.
\end{equation*}%
\begin{equation*}
\left. +\frac{1}{^{M^{\beta }}}+\omega \left( \varepsilon \right) \right)
+c_{6}c_{7}\frac{M^{p-1}}{\varepsilon ^{p-1}\mu ^{2}}\int\limits_{B\left(
0,r_{0}\right) }\left( \frac{\lambda }{a\left( x\right) +\lambda }\right)
^{2}dx\leq \gamma \text{.}
\end{equation*}%
Thus, by (3.10), (3.16) and (3.17), the proof of the lemma is complete.
\end{proof}

Now, we prove the following theorem on the asymptotic compactness of $%
\left\{ S(t)\right\} _{t\geq 0}$ in $H^{2}\left(
\mathbb{R}
^{n}\right) \times L^{2}\left(
\mathbb{R}
^{n}\right) $, which plays a key role in the existence of the global
attractor.

\begin{theorem}
Assume that the conditions (2.3)-(2.5) hold and $B$ is a bounded subset of$\
H^{2}\left(
\mathbb{R}
^{n}\right) \times L^{2}\left(
\mathbb{R}
^{n}\right) $. Then for every sequence of the form $\left\{ S(t_{k})\varphi
_{k}\right\} _{k=1}^{\infty },$ where $\left\{ \varphi _{k}\right\}
_{k=1}^{\infty }\subset B$, $t_{k}\rightarrow \infty ,$ \ has a convergent
subsequence in $H^{2}\left(
\mathbb{R}
^{n}\right) \times L^{2}\left(
\mathbb{R}
^{n}\right) $.
\end{theorem}

\begin{proof}
To get the claim of the theorem, it is sufficient to prove the following
sequential limit estimate
\begin{equation}
\liminf\limits_{k\rightarrow \infty }\liminf\limits_{m\rightarrow \infty
}\left\Vert S\left( t_{k}\right) \varphi _{k}-S\left( t_{m}\right) \varphi
_{m}\right\Vert _{H^{2}\left(
\mathbb{R}
^{n}\right) \times L^{2}\left(
\mathbb{R}
^{n}\right) }=0\text{,}  \tag{3.18}
\end{equation}%
for every $\left\{ \varphi _{k}\right\} _{k=1}^{\infty }\subset B$ and $%
t_{k}\rightarrow \infty $. Indeed, establishing (3.18) and using the
argument at the end of the proof of [12, Lemma 3.4], we obtain the desired
result.

Now, by (2.3), (2.5) and (2.6), we have
\begin{equation}
\underset{t\geq 0}{\sup }\underset{\varphi \in B}{\sup }\left\Vert S\left(
t\right) \varphi \right\Vert _{H^{2}\left(
\mathbb{R}
^{n}\right) \times L^{2}\left(
\mathbb{R}
^{n}\right) }<\infty .  \tag{3.19}
\end{equation}%
Since $\left\{ \varphi _{k}\right\} _{k=1}^{\infty }$ is bounded in $%
H^{2}\left(
\mathbb{R}
^{n}\right) \times L^{2}\left(
\mathbb{R}
^{n}\right) $, by (3.19), the sequence $\left\{ S\left( .\right) \varphi
_{k}\right\} _{k=1}^{\infty }$ is bounded in $C_{b}\left( 0,\infty
;H^{2}\left(
\mathbb{R}
^{n}\right) \times L^{2}\left(
\mathbb{R}
^{n}\right) \right) $, where $C_{b}\left( 0,\infty ;H^{2}\left(
\mathbb{R}
^{n}\right) \times L^{2}\left(
\mathbb{R}
^{n}\right) \right) $ is the space of continuously bounded functions from $%
\left[ 0,\infty \right) $ to $H^{2}\left(
\mathbb{R}
^{n}\right) \times L^{2}\left(
\mathbb{R}
^{n}\right) $. Then for any $T_{0}\geq 0$ \ there exists a subsequence $%
\left\{ k_{m}\right\} _{m=1}^{\infty }$ such that $t_{k_{m}}\geq T_{0}$, and
\begin{equation}
\left\{
\begin{array}{c}
v_{m}\rightarrow v\text{ weakly star in }L^{\infty }\left( 0,\infty
;H^{2}\left(
\mathbb{R}
^{n}\right) \right) \text{,} \\
v_{mt}\rightarrow v_{t}\text{ weakly star in }L^{\infty }\left( 0,\infty
;L^{2}\left(
\mathbb{R}
^{n}\right) \right) \text{,} \\
\left\Vert v_{m}\left( t\right) \right\Vert _{L^{p}\left(
\mathbb{R}
^{n}\right) }^{p}\rightarrow q\left( t\right) \text{ weakly star in }%
W^{1,\infty }\left( 0,\infty \right) \text{,} \\
v_{m}\left( t\right) \rightarrow v\left( t\right) \text{ weakly in }%
H^{2}\left(
\mathbb{R}
^{n}\right) \text{, }\forall t\geq 0\text{,}%
\end{array}%
\right.  \tag{3.20}
\end{equation}%
for some $q\in W^{1,\infty }\left( 0,\infty \right) $ and $v\in L^{\infty
}\left( 0,\infty ;H^{2}\left(
\mathbb{R}
^{n}\right) \right) \cap W^{1,\infty }\left( 0,\infty ;L^{2}\left(
\mathbb{R}
^{n}\right) \right) $, where $\left( v_{m}(t\right) ,v_{mt}\left( t\right)
)=S(t+t_{k_{m}}-T_{0})\varphi _{k_{m}}$. By (2.1)$,$ we also have%
\begin{equation*}
v_{mtt}(t,x)-v_{ltt}(t,x)+\Delta ^{2}\left( v_{m}(t,x)-v_{l}(t,x)\right)
+\alpha (x)\left( v_{mt}(t,x)-v_{lt}(t,x)\right) +\lambda \left(
v_{m}(t,x)-v_{l}(t,x)\right)
\end{equation*}%
\begin{equation}
=f(\left\Vert v_{l}\left( t\right) \right\Vert _{L^{p}\left(
\mathbb{R}
^{n}\right) }))\left\vert v_{l}(t,x)\right\vert
^{p-2}v_{l}(t,x)-f(\left\Vert v_{m}\left( t\right) \right\Vert _{L^{p}\left(
\mathbb{R}
^{n}\right) })\left\vert v_{m}(t,x)\right\vert ^{p-2}v_{m}(t,x).\newline
\tag{3.21}
\end{equation}

We obtain (3.18) by means of the sequential limit estimate of the energy of $%
v_{m}-v_{l}$ which is proved in the following three steps. In the first
step, we get the tail estimates, by using the effect of the damping term. In
the second step, we obtain the interior estimates. Finally, in the last
step, we get the sequential limit estimate of the energy in $%
\mathbb{R}
^{n}$, by considering the results obtained in the previous steps. Note that
we establish these estimates for the smooth solutions of (2.1)-(2.2) with
the initial data in $H^{4}\left(
\mathbb{R}
^{n}\right) \times H^{2}\left(
\mathbb{R}
^{n}\right) ,$ for which the estimates in the following text are justified.
These estimates can be extended to the weak solutions with the initial data
in $H^{2}\left(
\mathbb{R}
^{n}\right) \times L^{2}\left(
\mathbb{R}
^{n}\right) $ by the standard density arguments.\newline
\textbf{Step 1 (Tail estimates): }Taking into account (2.3), (2.4), (2.5)
and (2.6) we get%
\begin{equation}
\int\limits_{0}^{T}\left\Vert v_{mt}(t)\right\Vert _{L^{2}\left(
\mathbb{R}
^{n}\backslash B\left( 0,r_{0}\right) \right) }^{2}dt\leq c_{1}\text{, }%
\forall T\geq 0\text{.}  \tag{3.22}
\end{equation}%
Now, putting $v_{m}$ instead of $v$ in (2.1), we have%
\begin{equation*}
v_{mtt}(t,x)+\Delta ^{2}v_{m}(t,x)+\alpha (x)v_{mt}(t,x)+\lambda v_{m}(t,x)
\end{equation*}%
\begin{equation*}
+f(\left\Vert v_{m}(t)\right\Vert _{L^{p}\left(
\mathbb{R}
^{n}\right) })\left\vert v_{m}(t,x)\right\vert ^{p-2}v_{m}(t,x)=h\left(
x\right) \text{.}
\end{equation*}%
Let $\eta \in C^{\infty }\left(
\mathbb{R}
^{n}\right) $, $0\leq \eta \left( x\right) \leq 1$, $\eta \left( x\right)
=\left\{
\begin{array}{c}
0,\text{ }\left\vert x\right\vert \leq 1 \\
1,\text{ }\left\vert x\right\vert \geq 2%
\end{array}%
\right. $ and $\eta _{r}\left( x\right) =\eta \left( \frac{x}{r}\right) $.
Multiplying above equation by $\eta _{r}^{2}v_{m}$ and integrating over $%
\left( 0,T\right) \times
\mathbb{R}
^{n}$, we get%
\begin{equation*}
\int\limits_{0}^{T}\left( \left\Vert \eta _{r}\Delta v_{m}(t)\right\Vert
_{L^{2}\left(
\mathbb{R}
^{n}\right) }^{2}+\lambda \left\Vert \eta _{r}v_{m}(t)\right\Vert
_{L^{2}\left(
\mathbb{R}
^{n}\right) }^{2}\right) dt
\end{equation*}%
\begin{equation*}
=\int\limits_{0}^{T}\left\Vert \eta _{r}v_{mt}\left( t\right) \right\Vert
_{L^{2}\left(
\mathbb{R}
^{n}\right) }^{2}dt-\left. \left( \int\limits_{%
\mathbb{R}
^{n}}\eta _{r}^{2}\left( x\right) v_{mt}\left( t,x\right)
v_{m}(t,x)dx\right) \right\vert _{0}^{T}
\end{equation*}%
\begin{equation*}
-\frac{4}{r}\sum_{i=1}^{n}\int\limits_{0}^{T}\eta _{r}\left( x\right) \eta
_{x_{i}}\left( \frac{x}{r}\right) \Delta
v_{m}(t,x)v_{mx_{i}}(t,x)dxdt-\int\limits_{0}^{T}\int\limits_{%
\mathbb{R}
^{n}}\Delta \left( \eta _{r}^{2}\left( x\right) \right) \Delta
v_{m}(t,x)v_{m}(t,x)dxdt
\end{equation*}%
\begin{equation*}
-\frac{1}{2}\left. \left( \int\limits_{%
\mathbb{R}
^{n}}\eta _{r}^{2}\left( x\right) \alpha \left( x\right) \left(
v_{m}(t,x)\right) ^{2}dx\right) \right\vert _{0}^{T}
\end{equation*}%
\begin{equation*}
-\int\limits_{0}^{T}f(\left\Vert v_{m}(t)\right\Vert _{L^{p}\left(
\mathbb{R}
^{n}\right) })\int\limits_{%
\mathbb{R}
^{n}}\left\vert v_{m}(t,x)\right\vert ^{p}\eta _{r}^{2}\left( x\right)
dxdt+\int\limits_{0}^{T}\int\limits_{%
\mathbb{R}
^{n}}h\left( x\right) \eta _{r}^{2}\left( x\right) v_{m}\left( t,x\right)
dxdt.
\end{equation*}%
Taking into account (2.3), (2.5), (3.19) and (3.22) we obtain%
\begin{equation*}
\int\limits_{0}^{T}\left( \left\Vert \Delta \left( v_{m}(t)\right)
\right\Vert _{L^{2}\left(
\mathbb{R}
^{n}\backslash B\left( 0,2r\right) \right) }^{2}+\lambda \left\Vert
v_{m}(t)\right\Vert _{L^{2}\left(
\mathbb{R}
^{n}\backslash B\left( 0,2r\right) \right) }^{2}\right) dt
\end{equation*}%
\begin{equation}
\leq c_{2}\left( 1+\frac{T}{r}+T\left\Vert h\right\Vert _{L^{2}\left(
\mathbb{R}
^{n}\backslash B\left( 0,r\right) \right) }\right) \text{, \ \ }\forall
T\geq 0\text{ and }\forall r\geq r_{0}\text{.}  \tag{3.23}
\end{equation}%
\textbf{Step 2 (Interior estimates): }Multiplying (3.21) by $%
\sum\nolimits_{i=1}^{n}x_{i}\left( 1-\eta _{2r}\right) \left(
v_{m}-v_{l}\right) _{x_{i}}$\newline
$+\frac{1}{2}\left( n-1\right) \left( 1-\eta _{2r}\right) \left(
v_{m}-v_{l}\right) $, and integrating over $\left( 0,T\right) \times
\mathbb{R}
^{n}$, we find%
\begin{equation*}
\frac{3}{2}\int\limits_{0}^{T}\left\Vert \Delta \left( v_{m}\left( t\right)
-v_{l}\left( t\right) \right) \right\Vert _{L^{2}\left( B\left( 0,2r\right)
\right) }^{2}dt+\frac{1}{2}\int\limits_{0}^{T}\left\Vert v_{mt}\left(
t\right) -v_{lt}\left( t\right) \right\Vert _{L^{2}\left( B\left(
0,2r\right) \right) }^{2}dt
\end{equation*}%
\begin{equation*}
\leq \left\vert \sum\nolimits_{i=1}^{n}\left( \int\limits_{B\left(
0,4r\right) }\left( 1-\eta _{2r}\left( x\right) \right) x_{i}\left(
v_{m}(T,x)-v_{l}(T,x)\right) _{x_{i}}\left( v_{mt}(T,x)-v_{lt}(T,x)\right)
dx\right) \right\vert
\end{equation*}%
\begin{equation*}
+\left\vert \sum\nolimits_{i=1}^{n}\left( \int\limits_{B\left( 0,4r\right)
}\left( 1-\eta _{2r}\left( x\right) \right) x_{i}\left(
v_{m}(0,x)-v_{l}(0,x)\right) _{x_{i}}\left( v_{mt}(0,x)-v_{lt}(0,x)\right)
dx\right) \right\vert
\end{equation*}%
\begin{equation*}
+\frac{1}{2}\left( n-1\right) \left\vert \int\limits_{B\left( 0,4r\right)
}\left( 1-\eta _{2r}\left( x\right) \right) \left(
v_{mt}(T,x)-v_{lt}(T,x)\right) \left( v_{m}(T,x)-v_{l}(T,x)\right)
dx\right\vert
\end{equation*}%
\begin{equation*}
+\frac{1}{2}\left( n-1\right) \left\vert \int\limits_{B\left( 0,4r\right)
}\left( 1-\eta _{2r}\left( x\right) \right) \left(
v_{mt}(0,x)-v_{lt}(0,x)\right) \left( v_{m}(0,x)-v_{l}(0,x)\right)
dx\right\vert
\end{equation*}%
\begin{equation*}
+\frac{1}{4r}\left\vert
\sum\nolimits_{i=1}^{n}\int\limits_{0}^{T}\int\limits_{B\left( 0,4r\right)
\backslash B\left( 0,2r\right) }\eta _{x_{i}}\left( \frac{x}{2r}\right)
x_{i}\left( v_{mt}\left( t,x\right) -v_{lt}\left( t,x\right) \right)
^{2}dxdt\right\vert
\end{equation*}%
\begin{equation*}
+\frac{1}{4r}\left\vert
\sum\nolimits_{i=1}^{n}\int\limits_{0}^{T}\int\limits_{B\left( 0,4r\right)
\backslash B\left( 0,2r\right) }\eta _{x_{i}}\left( \frac{x}{2r}\right)
x_{i}\left( \Delta v_{m}\left( t,x\right) -\Delta v_{l}\left( t,x\right)
\right) ^{2}dxdt\right\vert
\end{equation*}%
\begin{equation*}
+\left\vert \sum\nolimits_{i=1}^{n}\int\limits_{0}^{T}\int\limits_{B\left(
0,4r\right) }\Delta \left( \left( 1-\eta _{2r}\left( x\right) \right)
x_{i}\right) \left( v_{m}\left( t,x\right) -v_{l}\left( t,x\right) \right)
_{x_{i}}\Delta \left( v_{m}\left( t,x\right) -v_{l}\left( t,x\right) \right)
dxdt\right\vert
\end{equation*}%
\begin{equation*}
+\frac{1}{r}\left\vert
\sum\nolimits_{i,j=1}^{n}\int\limits_{0}^{T}\int\limits_{B\left( 0,4r\right)
\backslash B\left( 0,2r\right) }\eta _{x_{j}}\left( \frac{x}{2r}\right)
x_{i}\left( v_{m}\left( t,x\right) -v_{l}\left( t,x\right) \right)
_{x_{i}x_{j}}\Delta \left( v_{m}\left( t,x\right) -v_{l}\left( t,x\right)
\right) dxdt\right\vert
\end{equation*}%
\begin{equation*}
+\frac{1}{2}\left( n-1\right) \left\vert
\int\limits_{0}^{T}\int\limits_{B\left( 0,4r\right) \backslash B\left(
0,2r\right) }\Delta \left( \left( 1-\eta _{2r}\left( x\right) \right)
\right) \left( v_{m}\left( t,x\right) -v_{l}\left( t,x\right) \right) \Delta
\left( v_{m}\left( t,x\right) -v_{l}\left( t,x\right) \right) dxdt\right\vert
\end{equation*}%
\begin{equation*}
+\frac{1}{2r}\left( n-1\right) \left\vert
\sum\nolimits_{i=1}^{n}\int\limits_{0}^{T}\int\limits_{B\left( 0,4r\right)
\backslash B\left( 0,2r\right) }\eta _{x_{i}}\left( \frac{x}{2r}\right)
\left( v_{m}\left( t,x\right) -v_{l}\left( t,x\right) \right) _{x_{i}}\Delta
\left( v_{m}\left( t,x\right) -v_{l}\left( t,x\right) \right) dxdt\right\vert
\end{equation*}%
\begin{equation*}
+\left\vert \sum\nolimits_{i=1}^{n}\int\limits_{0}^{T}\int\limits_{B\left(
0,4r\right) }\left( 1-\eta _{2r}\left( x\right) \right) x_{i}\left(
v_{m}\left( t,x\right) -v_{l}\left( t,x\right) \right) _{x_{i}}a\left(
x\right) \left( v_{mt}\left( t,x\right) -v_{lt}\left( t,x\right) \right)
dxdt\right\vert
\end{equation*}%
\begin{equation*}
+\frac{1}{2}\left( n-1\right) \left\vert
\int\limits_{0}^{T}\int\limits_{B\left( 0,4r\right) }\left( 1-\eta
_{2r}\left( x\right) \right) \left( v_{m}\left( t,x\right) -v_{l}\left(
t,x\right) \right) a\left( x\right) \left( v_{mt}\left( t,x\right)
-v_{lt}\left( t,x\right) \right) dxdt\right\vert
\end{equation*}%
\begin{equation*}
+\lambda \left\vert
\sum\nolimits_{i=1}^{n}\int\limits_{0}^{T}\int\limits_{B\left( 0,4r\right)
}\left( 1-\eta _{2r}\left( x\right) \right) x_{i}\left( v_{m}\left(
t,x\right) -v_{l}\left( t,x\right) \right) _{x_{i}}\left( v_{m}\left(
t,x\right) -v_{l}\left( t,x\right) \right) dxdt\right\vert
\end{equation*}%
\begin{equation*}
+\left\vert \sum\nolimits_{i=1}^{n}\int\limits_{0}^{T}\int\limits_{B\left(
0,4r\right) }\left( 1-\eta _{2r}\left( x\right) \right) x_{i}\left(
v_{m}\left( t,x\right) -v_{l}\left( t,x\right) \right) _{x_{i}}\right.
\end{equation*}%
\begin{equation*}
\left. \times \left( f(\left\Vert v_{m}\left( t\right) \right\Vert
_{L^{p}\left(
\mathbb{R}
^{n}\right) })\left\vert v_{m}(t,x)\right\vert ^{p-2}v_{m}(t,x)-f(\left\Vert
v_{l}\left( t\right) \right\Vert _{L^{p}\left(
\mathbb{R}
^{n}\right) })\left\vert v_{l}(t,x)\right\vert ^{p-2}v_{l}(t,x)\right)
dxdt\right\vert
\end{equation*}%
\begin{equation*}
+\frac{1}{2}\left( n-1\right) \left\vert
\int\limits_{0}^{T}\int\limits_{B\left( 0,4r\right) }\left( 1-\eta
_{2r}\left( x\right) \right) \left( v_{m}\left( t,x\right) -v_{l}\left(
t,x\right) \right) \right.
\end{equation*}%
\begin{equation*}
\left. \times \left( f(\left\Vert v_{m}\left( t\right) \right\Vert
_{L^{p}\left(
\mathbb{R}
^{n}\right) })\left\vert v_{m}(t,x)\right\vert ^{p-2}v_{m}(t,x)-f(\left\Vert
v_{l}\left( t\right) \right\Vert _{L^{p}\left(
\mathbb{R}
^{n}\right) })\left\vert v_{l}(t,x)\right\vert ^{p-2}v_{l}(t,x)\right)
dxdt\right\vert
\end{equation*}%
\begin{equation*}
\leq c_{3}r\left( \left\Vert \nabla v_{m}\left( T\right) -\nabla v_{l}\left(
T\right) \right\Vert _{L^{2}(B\left( 0,4r\right) )}+\left\Vert \nabla
v_{m}\left( 0\right) -\nabla v_{l}\left( 0\right) \right\Vert
_{L^{2}(B\left( 0,4r\right) )}\right)
\end{equation*}%
\begin{equation*}
c_{3}\left\Vert v_{mt}-v_{lt}\right\Vert _{L^{2}\left( 0,T;L^{2}\left(
B\left( 0,4r\right) \backslash B\left( 0,2r\right) \right) \right)
}^{2}+c_{3}\left\Vert v_{m}-v_{l}\right\Vert _{L^{2}\left( 0,T;H^{2}\left(
B\left( 0,4r\right) \backslash B\left( 0,2r\right) \right) \right) }^{2}
\end{equation*}%
\begin{equation}
+c_{3}r\sqrt{T}\left\Vert \nabla v_{m}-\nabla v_{l}\right\Vert
_{L^{2}(\left( 0,T\right) \times B\left( 0,4r\right) )},  \tag{3.24}
\end{equation}%
since, by (2.5) and (3.19),
\begin{equation*}
\left\Vert f(\left\Vert v_{m}\left( t\right) \right\Vert _{L^{p}\left(
\mathbb{R}
^{n}\right) })\left\vert v_{m}(t)\right\vert ^{p-2}v_{m}(t)-f(\left\Vert
v_{l}\left( t\right) \right\Vert _{L^{p}\left(
\mathbb{R}
^{n}\right) })\left\vert v_{l}(t)\right\vert ^{p-2}v_{l}(t)\right\Vert
_{L^{2}(B\left( 0,4r\right) )}\leq \widetilde{c}\text{.}
\end{equation*}%
Since the sequence $\left\{ v_{m}\right\} _{m=1}^{\infty }$ is bounded in $%
C\left( \left[ 0,T\right] ;H^{2}\left(
\mathbb{R}
^{n}\right) \right) $ and the sequence $\left\{ v_{mt}\right\}
_{m=1}^{\infty }$ is bounded in $C\left( \left[ 0,T\right] ;L^{2}\left(
\mathbb{R}
^{n}\right) \right) $, by the generalized Arzela-Ascoli theorem, the
sequence $\left\{ v_{m}\right\} _{m=1}^{\infty }$ $\ $is relatively compact
in $C\left( \left[ 0,T\right] ;H^{1}\left( B\left( 0,r\right) \right)
\right) $ for every $r>0$. So, according to (3.20)$_{1}$-(3.20)$_{2}$, the
sequence $\left\{ v_{m}\right\} _{m=1}^{\infty }$ strongly converges to $v$
in $C\left( \left[ 0,T\right] ;H^{1}\left( B\left( 0,r\right) \right)
\right) $. Then, by using (3.22) and (3.23) in (3.24) , we get
\begin{equation*}
\underset{m\rightarrow \infty }{\lim \sup }\text{ }\underset{l\rightarrow
\infty }{\lim \sup }\int\limits_{0}^{T}\left[ \left\Vert \Delta \left(
v_{m}\left( t\right) -v_{l}\left( t\right) \right) \right\Vert _{L^{2}\left(
B\left( 0,2r\right) \right) }^{2}+\left\Vert v_{mt}\left( t\right)
-v_{lt}\left( t\right) \right\Vert _{L^{2}\left( B\left( 0,2r\right) \right)
}^{2}\right] dt
\end{equation*}%
\begin{equation*}
\leq c_{4}\left( 1+\frac{T}{r}+T\left\Vert h\right\Vert _{L^{2}\left(
\mathbb{R}
^{n}\backslash B\left( 0,r\right) \right) }\right) \text{, \ \ }\forall
T\geq 0\text{ and }\forall r\geq r_{0}\text{.}
\end{equation*}%
\textbf{Step 3 (Estimates in }$%
\mathbb{R}
^{n}$\textbf{): }By using (3.22), (3.23) and the last estimate of the
previous step, we obtain%
\begin{equation*}
\underset{m\rightarrow \infty }{\lim \sup }\text{ }\underset{l\rightarrow
\infty }{\lim \sup }\int\limits_{0}^{T}\left[ \left\Vert v_{m}\left(
t\right) -v_{l}\left( t\right) \right\Vert _{H^{2}\left(
\mathbb{R}
^{n}\right) }^{2}+\left\Vert v_{mt}\left( t\right) -v_{lt}\left( t\right)
\right\Vert _{L^{2}\left(
\mathbb{R}
^{n}\right) }^{2}\right] dt
\end{equation*}%
\begin{equation*}
\leq c_{4}\left( 1+\frac{T}{r}+T\left\Vert h\right\Vert _{L^{2}\left(
\mathbb{R}
^{n}\backslash B\left( 0,r\right) \right) }\right) \text{, \ \ }\forall
T\geq 0\text{ and }\forall r\geq r_{0}\text{. }
\end{equation*}%
Passing to limit as $r\rightarrow \infty $ in the last inequality, we get
\begin{equation}
\underset{m\rightarrow \infty }{\lim \sup }\text{ }\underset{l\rightarrow
\infty }{\lim \sup }\int\limits_{0}^{T}\left[ \left\Vert v_{m}\left(
t\right) -v_{l}\left( t\right) \right\Vert _{H^{2}\left(
\mathbb{R}
^{n}\right) }^{2}+\left\Vert v_{mt}\left( t\right) -v_{lt}\left( t\right)
\right\Vert _{L^{2}\left(
\mathbb{R}
^{n}\right) }^{2}\right] dt\leq c_{5}\text{, }\forall T\geq 0.  \tag{3.25}
\end{equation}%
Multiplying (3.21) by $2t\left( v_{mt}-v_{lt}\right) $, integrating over $%
\left( 0,T\right) \times
\mathbb{R}
^{n}$, using integration by parts and considering (2.4), we find%
\begin{equation*}
T\left\Vert \Delta \left( v_{m}\left( T\right) -v_{l}\left( T\right) \right)
\right\Vert _{L^{2}\left(
\mathbb{R}
^{n}\right) }^{2}+T\left\Vert v_{mt}\left( T\right) -v_{lt}\left( T\right)
\right\Vert _{L^{2}\left(
\mathbb{R}
^{n}\right) }^{2}+T\lambda \left\Vert v_{m}\left( T\right) -v_{l}\left(
T\right) \right\Vert _{L^{2}\left(
\mathbb{R}
^{n}\right) }^{2}
\end{equation*}%
\begin{equation*}
+2\alpha _{0}\int\limits_{0}^{T}\int\limits_{%
\mathbb{R}
^{n}\backslash B\left( 0,r\right) }t\left( v_{mt}\left( t\right)
-v_{lt}\left( t\right) \right) ^{2}dxdt
\end{equation*}%
\begin{equation*}
\leq \int\limits_{0}^{T}\left\Vert v_{mt}\left( t\right) -v_{lt}\left(
t\right) \right\Vert _{L^{2}\left(
\mathbb{R}
^{n}\right) }^{2}dt+\int\limits_{0}^{T}\left\Vert \Delta \left(
v_{m}(t)-v_{l}(t)\right) \right\Vert _{L^{2}\left(
\mathbb{R}
^{n}\right) }^{2}dt+\lambda \int\limits_{0}^{T}\left\Vert
v_{m}(t)-v_{l}(t)\right\Vert _{L^{2}\left(
\mathbb{R}
^{n}\right) }^{2}dt
\end{equation*}%
\begin{equation*}
+2\int\limits_{0}^{T}\int\limits_{%
\mathbb{R}
^{n}}t\left( f(\left\Vert v_{l}\left( t\right) \right\Vert _{L^{p}\left(
\mathbb{R}
^{n}\right) })\left\vert v_{l}(t,x)\right\vert ^{p-2}v_{l}(t,x)-f(\left\Vert
v_{m}\left( t\right) \right\Vert _{L^{p}\left(
\mathbb{R}
^{n}\right) })\left\vert v_{m}(t,x)\right\vert ^{p-2}v_{m}(t,x)\right)
\end{equation*}%
\begin{equation}
\times \left( v_{mt}\left( t,x\right) -v_{lt}\left( t,x\right) \right) dxdt%
\text{.}  \tag{3.26}
\end{equation}%
Multiplying (3.21) by $t\eta _{r}\left( v_{m}-v_{l}\right) $, integrating
over $\left( 0,T\right) \times
\mathbb{R}
^{n}$ and using integration by parts, we get%
\begin{equation*}
T\int\limits_{%
\mathbb{R}
^{n}}\left( v_{mt}\left( T,x\right) -v_{lt}\left( T,x\right) \right) \eta
_{r}\left( x\right) \left( v_{m}\left( T,x\right) -v_{l}\left( T,x\right)
\right) dx-\int\limits_{0}^{T}\int\limits_{%
\mathbb{R}
^{n}}t\eta _{r}\left( x\right) \left( v_{mt}\left( t,x\right) -v_{lt}\left(
t,x\right) \right) ^{2}dxdt
\end{equation*}%
\begin{equation*}
-\int\limits_{0}^{T}\int\limits_{%
\mathbb{R}
^{n}}\eta _{r}\left( x\right) \left( v_{mt}\left( t,x\right) -v_{lt}\left(
t,x\right) \right) \left( v_{m}\left( t,x\right) -v_{l}\left( t,x\right)
\right) dxdt
\end{equation*}%
\begin{equation*}
+\int\limits_{0}^{T}\int\limits_{%
\mathbb{R}
^{n}}t\left( \Delta \left( v_{m}\left( t,x\right) -v_{l}\left( t,x\right)
\right) \right) ^{2}\eta _{r}\left( x\right) dxdt
\end{equation*}%
\begin{equation*}
+2\sum\nolimits_{i=1}^{n}\int\limits_{0}^{T}\int\limits_{%
\mathbb{R}
^{n}}\left( \Delta \left( v_{m}\left( t,x\right) -v_{l}\left( t,x\right)
\right) \right) t\left( \eta _{r}\left( x\right) \right) _{x_{i}}\left(
v_{m}\left( t,x\right) -v_{l}\left( t,x\right) \right) _{x_{i}}dxdt
\end{equation*}%
\begin{equation*}
+\int\limits_{0}^{T}\int\limits_{%
\mathbb{R}
^{n}}\left( \Delta \left( v_{m}\left( t,x\right) -v_{l}\left( t,x\right)
\right) \right) t\Delta \left( \eta _{r}\left( x\right) \right) \left(
v_{m}\left( t,x\right) -v_{l}\left( t,x\right) \right) dxdt
\end{equation*}%
\begin{equation*}
+\frac{T}{2}\int\limits_{%
\mathbb{R}
^{n}}\alpha \left( x\right) \left( v_{m}\left( T,x\right) -v_{l}\left(
T,x\right) \right) ^{2}\eta _{r}\left( x\right) dx
\end{equation*}%
\begin{equation*}
-\frac{1}{2}\int\limits_{0}^{T}\int\limits_{%
\mathbb{R}
^{n}}\alpha \left( x\right) \eta _{r}\left( x\right) \left( v_{m}\left(
t,x\right) -v_{l}\left( t,x\right) \right) ^{2}dxdt+\lambda
\int\limits_{0}^{T}\int\limits_{%
\mathbb{R}
^{n}}t\left( v_{m}\left( t,x\right) -v_{l}\left( t,x\right) \right) ^{2}\eta
_{r}\left( x\right) dxdt
\end{equation*}%
\begin{equation*}
+\int\limits_{0}^{T}\int\limits_{%
\mathbb{R}
^{n}}t\left( f(\left\Vert v_{m}\left( t\right) \right\Vert _{L^{p}\left(
\mathbb{R}
^{n}\right) })\left\vert v_{m}(t,x)\right\vert ^{p-2}v_{m}(t,x)-f(\left\Vert
v_{l}\left( t\right) \right\Vert _{L^{p}\left(
\mathbb{R}
^{n}\right) })\left\vert v_{l}(t,x)\right\vert ^{p-2}v_{l}(t,x)\right)
\end{equation*}%
\begin{equation*}
\times \eta _{r}\left( x\right) \left( v_{m}\left( t,x\right) -v_{l}\left(
t,x\right) \right) dxdt=0.
\end{equation*}%
Then, considering (2.3), we obtain%
\begin{equation*}
\int\limits_{0}^{T}\int\limits_{%
\mathbb{R}
^{n}}t\left( \Delta \left( v_{m}\left( t,x\right) -v_{l}\left( t,x\right)
\right) \right) ^{2}\eta _{r}\left( x\right) dxdt+\lambda
\int\limits_{0}^{T}\int\limits_{%
\mathbb{R}
^{n}}t\left( v_{m}\left( t,x\right) -v_{l}\left( t,x\right) \right) ^{2}\eta
_{r}\left( x\right) dxdt
\end{equation*}%
\begin{equation*}
\leq -T\int\limits_{%
\mathbb{R}
^{n}}\left( v_{mt}\left( T,x\right) -v_{lt}\left( T,x\right) \right) \eta
_{r}\left( x\right) \left( v_{m}\left( T,x\right) -v_{l}\left( T,x\right)
\right) dx
\end{equation*}%
\begin{equation*}
+\int\limits_{0}^{T}\int\limits_{%
\mathbb{R}
^{n}}t\eta _{r}\left( x\right) \left( v_{mt}\left( t,x\right) -v_{lt}\left(
t,x\right) \right) ^{2}dxdt
\end{equation*}%
\begin{equation*}
+\int\limits_{0}^{T}\int\limits_{%
\mathbb{R}
^{n}}\eta _{r}\left( x\right) \left( v_{mt}\left( t,x\right) -v_{lt}\left(
t,x\right) \right) \left( v_{m}\left( t,x\right) -v_{l}\left( t,x\right)
\right) dxdt
\end{equation*}%
\begin{equation*}
-2\sum\nolimits_{i=1}^{n}\int\limits_{0}^{T}\int\limits_{%
\mathbb{R}
^{n}}\left( \Delta \left( v_{m}\left( t,x\right) -v_{l}\left( t,x\right)
\right) \right) t\left( \eta _{r}\left( x\right) \right) _{x_{i}}\left(
v_{m}\left( t,x\right) -v_{l}\left( t,x\right) \right) _{x_{i}}dxdt
\end{equation*}%
\begin{equation*}
-\int\limits_{0}^{T}\int\limits_{%
\mathbb{R}
^{n}}\left( \Delta \left( v_{m}\left( t,x\right) -v_{l}\left( t,x\right)
\right) \right) t\Delta \left( \eta _{r}\left( x\right) \right) \left(
v_{m}\left( t,x\right) -v_{l}\left( t,x\right) \right) dxdt
\end{equation*}%
\begin{equation*}
+\frac{1}{2}\int\limits_{0}^{T}\int\limits_{%
\mathbb{R}
^{n}}\alpha \left( x\right) \eta _{r}\left( x\right) \left( v_{m}\left(
t,x\right) -v_{l}\left( t,x\right) \right) ^{2}dxdt
\end{equation*}%
\begin{equation*}
-\int\limits_{0}^{T}\int\limits_{%
\mathbb{R}
^{n}}t\left( f(\left\Vert v_{m}\left( t\right) \right\Vert _{L^{p}\left(
\mathbb{R}
^{n}\right) })\left\vert v_{m}(t,x)\right\vert ^{p-2}v_{m}(t,x)-f(\left\Vert
v_{l}\left( t\right) \right\Vert _{L^{p}\left(
\mathbb{R}
^{n}\right) })\left\vert v_{l}(t,x)\right\vert ^{p-2}v_{l}(t,x)\right)
\end{equation*}%
\begin{equation*}
\times \eta _{r}\left( x\right) \left( v_{m}\left( t,x\right) -v_{l}\left(
t,x\right) \right) dxdt,\text{ \ \ }\forall T\geq 0\text{ and }\forall r\geq
r_{0}\text{.}
\end{equation*}%
Taking into account (2.5) and (3.19) in the above inequality, we find%
\begin{equation*}
\int\limits_{0}^{T}\int\limits_{%
\mathbb{R}
^{n}}t\left( \Delta \left( v_{m}\left( t,x\right) -v_{l}\left( t,x\right)
\right) \right) ^{2}\eta _{r}\left( x\right) dxdt+\lambda
\int\limits_{0}^{T}\int\limits_{%
\mathbb{R}
^{n}}t\left( v_{m}\left( t,x\right) -v_{l}\left( t,x\right) \right) ^{2}\eta
_{r}\left( x\right) dxdt
\end{equation*}%
\begin{equation*}
\leq T\left( \left\Vert v_{mt}\left( T,x\right) -v_{lt}\left( T,x\right)
\right\Vert _{L^{2}\left(
\mathbb{R}
^{n}\right) }^{2}+\left\Vert v_{m}\left( T,x\right) -v_{l}\left( T,x\right)
\right\Vert _{L^{2}\left(
\mathbb{R}
^{n}\right) }^{2}\right)
\end{equation*}%
\begin{equation*}
+\int\limits_{0}^{T}\int\limits_{%
\mathbb{R}
^{n}}t\eta _{r}\left( x\right) \left( v_{mt}\left( t,x\right) -v_{lt}\left(
t,x\right) \right) ^{2}dxdt+\int\limits_{0}^{T}\int\limits_{%
\mathbb{R}
^{n}}\eta _{r}\left( x\right) \left( v_{mt}\left( t,x\right) -v_{lt}\left(
t,x\right) \right) ^{2}dxdt
\end{equation*}%
\begin{equation*}
+\int\limits_{0}^{T}\int\limits_{%
\mathbb{R}
^{n}}\eta _{r}\left( x\right) \left( v_{m}\left( t,x\right) -v_{l}\left(
t,x\right) \right) ^{2}dxdt+\frac{1}{2}\int\limits_{0}^{T}\int\limits_{%
\mathbb{R}
^{n}}\alpha \left( x\right) \eta _{r}\left( x\right) \left( v_{m}\left(
t,x\right) -v_{l}\left( t,x\right) \right) ^{2}dxdt
\end{equation*}%
\begin{equation}
+c_{6}\frac{T}{r}+\widetilde{K_{r}}^{m,l}\left( T\right) ,\text{ \ \ }%
\forall T\geq 0\text{ and }\forall r\geq r_{0},  \tag{3.27}
\end{equation}%
where%
\begin{equation*}
\widetilde{K_{r}}^{m,l}\left( T\right) :=\int\limits_{0}^{T}t\left(
f(\left\Vert v_{l}\left( t\right) \right\Vert _{L^{p}\left(
\mathbb{R}
^{n}\right) }-f(\left\Vert v_{m}\left( t\right) \right\Vert _{L^{p}\left(
\mathbb{R}
^{n}\right) }\right) \int\limits_{%
\mathbb{R}
^{n}}\left\vert v_{l}(t,x)\right\vert ^{p-2}v_{l}(t,x)\eta _{r}\left(
x\right)
\end{equation*}%
\begin{equation*}
\times \left( v_{m}\left( t,x\right) -v_{l}\left( t,x\right) \right) dxdt%
\text{,}
\end{equation*}%
and considering (3.19) -(3.20)$_{3}$, it is easy to see that
\begin{equation*}
\underset{m,l}{\sup }\left\Vert \widetilde{K_{r}}^{m,l}\right\Vert _{C\left[
0,T\right] }<\infty \text{ \ and }\underset{m\rightarrow \infty }{\lim }%
\underset{l\rightarrow \infty }{\lim \sup }\left\vert \widetilde{K_{r}}%
^{m,l}\left( T\right) \right\vert =0,\text{ \ }\forall T\geq 0.
\end{equation*}

Now, multiplying (3.27) by $\delta >0$ and adding to (3.26), we have%
\begin{equation*}
T\left\Vert \Delta \left( v_{m}\left( T\right) -v_{l}\left( T\right) \right)
\right\Vert _{L^{2}\left(
\mathbb{R}
^{n}\right) }^{2}+T\left\Vert v_{mt}\left( T\right) -v_{lt}\left( T\right)
\right\Vert _{L^{2}\left(
\mathbb{R}
^{n}\right) }^{2}+T\lambda \left\Vert v_{m}\left( T\right) -v_{l}\left(
T\right) \right\Vert _{L^{2}\left(
\mathbb{R}
^{n}\right) }^{2}
\end{equation*}%
\begin{equation*}
+2\alpha _{0}\int\limits_{0}^{T}\int\limits_{%
\mathbb{R}
^{n}\backslash B\left( 0,r\right) }t\left( v_{mt}\left( t\right)
-v_{lt}\left( t\right) \right) ^{2}dxdt
\end{equation*}%
\begin{equation*}
+\delta \int\limits_{0}^{T}\int\limits_{%
\mathbb{R}
^{n}}t\left( \Delta \left( v_{m}\left( t,x\right) -v_{l}\left( t,x\right)
\right) \right) ^{2}\eta _{r}\left( x\right) dxdt+\delta \lambda
\int\limits_{0}^{T}\int\limits_{%
\mathbb{R}
^{n}}t\left( v_{m}\left( t,x\right) -v_{l}\left( t,x\right) \right) ^{2}\eta
_{r}\left( x\right) dxdt
\end{equation*}%
\begin{equation*}
\leq \int\limits_{0}^{T}\left\Vert v_{mt}\left( T\right) -v_{lt}\left(
T\right) \right\Vert _{L^{2}\left(
\mathbb{R}
^{n}\right) }^{2}dt+\int\limits_{0}^{T}\left\Vert \Delta \left(
v_{m}(t)-v_{l}(t)\right) \right\Vert _{L^{2}\left(
\mathbb{R}
^{n}\right) }^{2}dt+\lambda \int\limits_{0}^{T}\left\Vert
v_{m}(t)-v_{l}(t)\right\Vert _{L^{2}\left(
\mathbb{R}
^{n}\right) }^{2}dt
\end{equation*}%
\begin{equation*}
+2\int\limits_{0}^{T}\int\limits_{%
\mathbb{R}
^{n}}t\left( f(\left\Vert v_{l}\left( t\right) \right\Vert _{L^{p}\left(
\mathbb{R}
^{n}\right) })\left\vert v_{l}(t,x)\right\vert ^{p-2}v_{l}(t,x)-f(\left\Vert
v_{m}\left( t\right) \right\Vert _{L^{p}\left(
\mathbb{R}
^{n}\right) })\left\vert v_{m}(t,x)\right\vert ^{p-2}v_{m}(t,x)\right)
\end{equation*}%
\begin{equation*}
\times \left( v_{mt}\left( t,x\right) -v_{lt}\left( t,x\right) \right) dxdt
\end{equation*}%
\begin{equation*}
+\delta T\left( \left\Vert v_{mt}\left( T,x\right) -v_{lt}\left( T,x\right)
\right\Vert _{L^{2}\left(
\mathbb{R}
^{n}\right) }^{2}+\left\Vert v_{m}\left( T,x\right) -v_{l}\left( T,x\right)
\right\Vert _{L^{2}\left(
\mathbb{R}
^{n}\right) }^{2}\right)
\end{equation*}%
\begin{equation*}
+\delta \int\limits_{0}^{T}\int\limits_{%
\mathbb{R}
^{n}}t\eta _{r}\left( x\right) \left( v_{mt}\left( t,x\right) -v_{lt}\left(
t,x\right) \right) ^{2}dxdt+\delta \int\limits_{0}^{T}\int\limits_{%
\mathbb{R}
^{n}}\eta _{r}\left( x\right) \left( v_{mt}\left( t,x\right) -v_{lt}\left(
t,x\right) \right) ^{2}dxdt
\end{equation*}%
\begin{equation*}
+\delta \int\limits_{0}^{T}\int\limits_{%
\mathbb{R}
^{n}}\eta _{r}\left( x\right) \left( v_{m}\left( t,x\right) -v_{l}\left(
t,x\right) \right) ^{2}dxdt+\frac{\delta }{2}\int\limits_{0}^{T}\int\limits_{%
\mathbb{R}
^{n}}\alpha \left( x\right) \eta _{r}\left( x\right) \left( v_{m}\left(
t,x\right) -v_{l}\left( t,x\right) \right) ^{2}dxdt
\end{equation*}%
\begin{equation}
+c_{6}\delta \frac{T}{r}+\delta \widetilde{K_{r}}^{m,l}\left( T\right) \text{%
, \ \ }\forall T\geq 0\text{ and }\forall r\geq r_{0}.  \tag{3.28}
\end{equation}%
Considering Lemma 3.2 in (3.28), for every $\gamma >0$, we get%
\begin{equation*}
T\left\Vert \Delta \left( v_{m}\left( T\right) -v_{l}\left( T\right) \right)
\right\Vert _{L^{2}\left(
\mathbb{R}
^{n}\right) }^{2}+T\left\Vert v_{mt}\left( T\right) -v_{lt}\left( T\right)
\right\Vert _{L^{2}\left(
\mathbb{R}
^{n}\right) }^{2}+T\lambda \left\Vert v_{m}\left( T\right) -v_{l}\left(
T\right) \right\Vert _{L^{2}\left(
\mathbb{R}
^{n}\right) }^{2}
\end{equation*}%
\begin{equation*}
+2\alpha _{0}\int\limits_{0}^{T}\int\limits_{%
\mathbb{R}
^{n}\backslash B\left( 0,r\right) }t\left( v_{mt}\left( t\right)
-v_{lt}\left( t\right) \right) ^{2}dxdt
\end{equation*}%
\begin{equation*}
+\delta \int\limits_{0}^{T}\int\limits_{%
\mathbb{R}
^{n}}t\left( \Delta \left( v_{m}\left( t,x\right) -v_{l}\left( t,x\right)
\right) \right) ^{2}\eta _{r}\left( x\right) dxdt+\delta \lambda
\int\limits_{0}^{T}\int\limits_{%
\mathbb{R}
^{n}}t\left( v_{m}\left( t,x\right) -v_{l}\left( t,x\right) \right) ^{2}\eta
_{r}\left( x\right) dxdt
\end{equation*}%
\begin{equation*}
\leq \int\limits_{0}^{T}\left\Vert v_{mt}\left( t\right) -v_{lt}\left(
t\right) \right\Vert _{L^{2}\left(
\mathbb{R}
^{n}\right) }^{2}dt+\int\limits_{0}^{T}\left\Vert \Delta \left(
v_{m}(t)-v_{l}(t)\right) \right\Vert _{L^{2}\left(
\mathbb{R}
^{n}\right) }^{2}dt+\lambda \int\limits_{0}^{T}\left\Vert
v_{m}(t)-v_{l}(t)\right\Vert _{L^{2}\left(
\mathbb{R}
^{n}\right) }^{2}dt
\end{equation*}%
\begin{equation*}
+\gamma \int\limits_{0}^{T}\tau E_{%
\mathbb{R}
^{n}\backslash \left( B\left( 0,2r\right) \right) }\left( v_{m}\left(
t\right) -v_{l}\left( t\right) \right) dt+c_{\gamma }\int\limits_{0}^{T}E_{%
\mathbb{R}
^{n}\backslash \left( B\left( 0,2r\right) \right) }\left( v_{m}\left(
t\right) -v_{l}\left( t\right) \right) dt
\end{equation*}%
\begin{equation*}
+c_{\gamma }\int\limits_{0}^{T}t\left( \left\Vert \sqrt{a}v_{mt}\left(
t\right) \right\Vert _{L^{2}\left(
\mathbb{R}
^{n}\right) }^{2}+\left\Vert \sqrt{a}v_{lt}\left( t\right) \right\Vert
_{L^{2}\left(
\mathbb{R}
^{n}\right) }^{2}\right) E_{%
\mathbb{R}
^{n}\backslash B\left( 0,2r\right) }\left( v_{m}\left( t\right) -v_{l}\left(
t\right) \right) dt+\left\vert K_{r}^{m,l}(T)\right\vert
\end{equation*}%
\begin{equation*}
+\delta T\left( \left\Vert v_{mt}\left( T,x\right) -v_{lt}\left( T,x\right)
\right\Vert _{L^{2}\left(
\mathbb{R}
^{n}\right) }^{2}+\left\Vert v_{m}\left( T,x\right) -v_{l}\left( T,x\right)
\right\Vert _{L^{2}\left(
\mathbb{R}
^{n}\right) }^{2}\right)
\end{equation*}%
\begin{equation*}
+\delta \int\limits_{0}^{T}\int\limits_{%
\mathbb{R}
^{n}}t\eta _{r}\left( x\right) \left( v_{mt}\left( t,x\right) -v_{lt}\left(
t,x\right) \right) ^{2}dxdt+\delta \int\limits_{0}^{T}\int\limits_{%
\mathbb{R}
^{n}}\eta _{r}\left( x\right) \left( v_{mt}\left( t,x\right) -v_{lt}\left(
t,x\right) \right) ^{2}dxdt
\end{equation*}%
\begin{equation*}
+\delta \int\limits_{0}^{T}\int\limits_{%
\mathbb{R}
^{n}}\eta _{r}\left( x\right) \left( v_{m}\left( t,x\right) -v_{l}\left(
t,x\right) \right) ^{2}dxdt+\frac{\delta }{2}\int\limits_{0}^{T}\int\limits_{%
\mathbb{R}
^{n}}\alpha \left( x\right) \eta _{r}\left( x\right) \left( v_{m}\left(
t,x\right) -v_{l}\left( t,x\right) \right) ^{2}dxdt
\end{equation*}%
\begin{equation*}
+c_{6}\delta \frac{T}{r}+\delta \widetilde{K_{r}}^{m,l}\left( T\right) \text{%
, \ \ }\forall T\geq 0\text{ and }\forall r\geq r_{0}.
\end{equation*}%
Then, for sufficiently small $\gamma $ and $\delta ,$ we obtain%
\begin{equation*}
TE_{%
\mathbb{R}
^{n}}\left( v_{m}\left( T\right) -v_{l}\left( T\right) \right) \leq
c_{7}\int\limits_{0}^{T}E_{%
\mathbb{R}
^{n}}\left( v_{m}\left( t\right) -v_{l}\left( t\right) \right) dt
\end{equation*}%
\begin{equation*}
+c_{\gamma }\int\limits_{0}^{T}t\left( \left\Vert \sqrt{a}v_{mt}\left(
t\right) \right\Vert _{L^{2}\left(
\mathbb{R}
^{n}\right) }^{2}+\left\Vert \sqrt{a}v_{lt}\left( t\right) \right\Vert
_{L^{2}\left(
\mathbb{R}
^{n}\right) }^{2}\right) E_{%
\mathbb{R}
^{n}}\left( v_{m}\left( t\right) -v_{l}\left( t\right) \right) dt
\end{equation*}%
\begin{equation*}
+c_{7}\left( \frac{T}{r}+\left\vert K_{r}^{m,l}\left( T\right) \right\vert
+\left\vert \widetilde{K_{r}}^{m,l}\left( T\right) \right\vert \right) \text{%
, \ \ }\forall T\geq 0\text{ and }\forall r\geq r_{0}\text{.}
\end{equation*}%
Now, denoting $y_{m,l}\left( t\right) :=tE_{%
\mathbb{R}
^{n}}\left( v_{m}\left( t\right) -v_{l}\left( t\right) \right) ,$ from the
previous inequality, we have%
\begin{equation*}
y_{m,l}\left( T\right) \leq c_{\gamma }\int\limits_{0}^{T}\left( \left\Vert
\sqrt{a}v_{mt}\left( t\right) \right\Vert _{L^{2}\left(
\mathbb{R}
^{n}\right) }^{2}+\left\Vert \sqrt{a}v_{lt}\left( t\right) \right\Vert
_{L^{2}\left(
\mathbb{R}
^{n}\right) }^{2}\right) y_{m,l}\left( t\right) dt
\end{equation*}%
\begin{equation*}
+c_{7}\int\limits_{0}^{T}E_{%
\mathbb{R}
^{n}}\left( v_{m}\left( t\right) -v_{l}\left( t\right) \right)
dt+c_{7}\left( \frac{T}{r}+\left\vert K_{r}^{m,l}\left( T\right) \right\vert
+\left\vert \widetilde{K_{r}}^{m,l}\left( T\right) \right\vert \right) \text{%
, \ }\forall T\geq 0\text{ and }\forall r\geq r_{0}\text{.}
\end{equation*}%
Applying Gronwall inequality and considering (2.6) and (3.19) in the above
inequality, we get%
\begin{equation*}
TE_{%
\mathbb{R}
^{n}}\left( v_{m}\left( T\right) -v_{l}\left( T\right) \right)
\end{equation*}%
\begin{equation*}
\leq c_{7}\int\limits_{0}^{T}E_{%
\mathbb{R}
^{n}}\left( v_{m}\left( t\right) -v_{l}\left( t\right) \right)
dt+c_{7}\left( \frac{T}{r}+\left\vert K_{r}^{m,l}\left( T\right) \right\vert
+\left\vert \widetilde{K_{r}}^{m,l}\left( T\right) \right\vert \right)
\end{equation*}%
\begin{equation*}
+c_{7}\int\limits_{0}^{T}\left( \int\limits_{0}^{t}E_{%
\mathbb{R}
^{n}}\left( v_{m}\left( s\right) -v_{l}\left( s\right) \right) ds+\frac{t}{r}%
+\left\vert K_{r}^{m,l}\left( t\right) \right\vert +\left\vert \widetilde{%
K_{r}}^{m,l}\left( t\right) \right\vert \right)
\end{equation*}%
\begin{equation*}
\times \left( \left\Vert \sqrt{a}v_{mt}\left( t\right) \right\Vert
_{L^{2}\left(
\mathbb{R}
^{n}\right) }^{2}+\left\Vert \sqrt{a}v_{lt}\left( t\right) \right\Vert
_{L^{2}\left(
\mathbb{R}
^{n}\right) }^{2}\right) e^{c_{\gamma }\int\nolimits_{t}^{T}\left(
\left\Vert \sqrt{a}v_{mt}\left( \tau \right) \right\Vert _{L^{2}\left(
\mathbb{R}
^{n}\right) }^{2}+\left\Vert \sqrt{a}v_{lt}\left( \tau \right) \right\Vert
_{L^{2}\left(
\mathbb{R}
^{n}\right) }^{2}\right) d\tau }ds
\end{equation*}%
\begin{equation*}
\leq c_{7}\int\limits_{0}^{T}E_{%
\mathbb{R}
^{n}}\left( v_{m}\left( t\right) -v_{l}\left( t\right) \right)
dt+c_{7}\left( \frac{T}{r}+\left\vert K_{r}^{m,l}\left( T\right) \right\vert
+\left\vert \widetilde{K_{r}}^{m,l}\left( T\right) \right\vert \right)
\end{equation*}%
\begin{equation*}
+c_{8}\int\limits_{0}^{T}E_{%
\mathbb{R}
^{n}}\left( v_{m}\left( t\right) -v_{l}\left( t\right) \right)
dt\int\limits_{0}^{T}\left( \left\Vert \sqrt{a}v_{mt}\left( t\right)
\right\Vert _{L^{2}\left(
\mathbb{R}
^{n}\right) }^{2}+\left\Vert \sqrt{a}v_{lt}\left( t\right) \right\Vert
_{L^{2}\left(
\mathbb{R}
^{n}\right) }^{2}\right)
\end{equation*}%
\begin{equation*}
+c_{8}\frac{T}{r}\int\limits_{0}^{T}\left( \left\Vert \sqrt{a}v_{mt}\left(
t\right) \right\Vert _{L^{2}\left(
\mathbb{R}
^{n}\right) }^{2}+\left\Vert \sqrt{a}v_{lt}\left( t\right) \right\Vert
_{L^{2}\left(
\mathbb{R}
^{n}\right) }^{2}\right) dt
\end{equation*}%
\begin{equation*}
+c_{8}\int\limits_{0}^{T}\left( \left\vert K_{r}^{m,l}\left( t\right)
\right\vert +\left\vert \widetilde{K_{r}}^{m,l}\left( t\right) \right\vert
\right) \left( \left\Vert \sqrt{a}v_{mt}\left( t\right) \right\Vert
_{L^{2}\left(
\mathbb{R}
^{n}\right) }^{2}+\left\Vert \sqrt{a}v_{lt}\left( t\right) \right\Vert
_{L^{2}\left(
\mathbb{R}
^{n}\right) }^{2}\right) dt
\end{equation*}%
\begin{equation*}
\leq c_{7}\left( \left\vert K_{r}^{m,l}\left( T\right) \right\vert
+\left\vert \widetilde{K_{r}}^{m,l}\left( T\right) \right\vert \right)
+c_{9}\int\limits_{0}^{T}E_{%
\mathbb{R}
^{n}}\left( v_{m}\left( t\right) -v_{l}\left( t\right) \right) dt+c_{9}\frac{%
T}{r}
\end{equation*}%
\begin{equation*}
+c_{9}\int\limits_{0}^{T}\left( \left\vert K_{r}^{m,l}\left( t\right)
\right\vert +\left\vert \widetilde{K_{r}}^{m,l}\left( t\right) \right\vert
\right) dt\text{, \ \ \ }\forall T\geq 0\text{ and }\forall r\geq r_{0}\text{%
.}
\end{equation*}%
By using Lebesgue dominated convergence theorem and considering (3.25) in
the last inequality, we obtain%
\begin{equation*}
\underset{m\rightarrow \infty }{\lim \sup }\text{ }\underset{l\rightarrow
\infty }{\lim \sup }\text{ }TE_{%
\mathbb{R}
^{n}}\left( v_{m}\left( T\right) -v_{l}\left( T\right) \right) \leq c_{10}(1+%
\frac{T}{r})\text{, \ }\forall T\geq 0\text{ and }\forall r\geq r_{0}\text{.}
\end{equation*}%
By passing to limit as $r\rightarrow \infty $ in the above inequality, we
find%
\begin{equation*}
\underset{m\rightarrow \infty }{\lim \sup }\text{ }\underset{l\rightarrow
\infty }{\lim \sup }\text{ }TE_{%
\mathbb{R}
^{n}}\left( v_{m}\left( T\right) -v_{l}\left( T\right) \right) \leq c_{10}%
\text{, \ }\forall T\geq 0,
\end{equation*}%
which gives
\begin{equation*}
\underset{m\rightarrow \infty }{\lim \sup }\text{ }\underset{l\rightarrow
\infty }{\lim \sup }\left\Vert S(T+t_{k_{m}}-T_{0})\varphi
_{k_{m}}-S(t+t_{k_{l}}-T_{0})\varphi _{k_{l}}\right\Vert _{H^{2}\left(
\mathbb{R}
^{n}\right) \times L^{2}\left(
\mathbb{R}
^{n}\right) }\leq \frac{c_{11}}{\sqrt{T}}\text{, \ }\forall T>0.
\end{equation*}%
Choosing $T=T_{0}$ in the previous inequality, we have%
\begin{equation*}
\underset{m\rightarrow \infty }{\lim \sup }\text{ }\underset{l\rightarrow
\infty }{\lim \sup }\left\Vert S(t_{k_{m}})\varphi
_{k_{m}}-S(t_{k_{l}})\varphi _{k_{l}}\right\Vert _{H^{2}\left(
\mathbb{R}
^{n}\right) \times L^{2}\left(
\mathbb{R}
^{n}\right) }\leq \frac{c_{11}}{\sqrt{T_{0}}}\text{, \ }\forall T_{0}>0.
\end{equation*}%
As a consequence, from the above sequential limit inequality, we get (3.18)
which completes the proof.
\end{proof}

Now we are in a position to complete the proof of the Theorem 2.2. Since, by
(2.3) and (2.6), problem (2.1)-(2.2) admits a strict Lyapunov function%
\begin{equation*}
\Phi \left( u\left( t\right) \right) =E_{%
\mathbb{R}
^{n}}\left( u\left( t\right) \right) +\frac{1}{p}F\left( \left\Vert u\left(
t\right) \right\Vert _{L^{p}\left(
\mathbb{R}
^{n}\right) }^{p}\right) -\int\limits_{%
\mathbb{R}
^{n}}h\left( x\right) u\left( t,x\right) dx\text{,}
\end{equation*}%
applying [6, Corollary 7.5.7], we obtain the claim of Theorem 2.2.

\end{document}